\documentclass[11pt,UTF8]{article}
\usepackage{bbm}
\usepackage[all,cmtip]{xy}
\usepackage{amsfonts,amssymb,amsmath,amscd,amsthm,mathtools}
\usepackage{mathrsfs}
\usepackage{latexsym}
\usepackage{color}
\usepackage{verbatim,enumitem}
\usepackage{hyperref}
\usepackage{tikz-cd}
\usepackage[OT1]{fontenc}

\input diagxy

\usepackage[total={155mm,225mm}]{geometry}

\DeclareMathOperator{\idl}{Idl}
\DeclareMathOperator{\with}{\&}

\DeclareMathOperator{\thda}{{\rotatebox[origin=c]{-90}{$\twoheadrightarrow$}}}

\theoremstyle{plain}
  \newtheorem{thm}{Theorem}[section]
  \newtheorem{lem}[thm]{Lemma}
  \newtheorem{prop}[thm]{Proposition}
  \newtheorem{cor}[thm]{Corollary}
\theoremstyle{definition}
  \newtheorem{defn}[thm]{Definition}
  \newtheorem{ques}[thm]{Question}
  \newtheorem{exmp}[thm]{Example}
  \newtheorem{rem}[thm]{Remark}

\newtheorem*{SA}{Standing Assumption and Convention}

\newcommand{\ra}{\rightarrow}
\newcommand{\lra}{\longrightarrow}

\newcommand{\conlat}{[0,1]\text{-}{\sf ConLat}}
\newcommand{\QSup}{\sQ\text{-}{\sf Sup}}

\newcommand{\QOrd}{\sQ\text{-}{\sf Ord}}

\newcommand{\sub}{{\rm sub}}
\newcommand{\id}{{\rm id}}

\newcommand{\CP}{\mathcal{P}}

\newcommand{\sfm}{{\sf m}}

\newcommand{\sy}{{\sf y}}

\newcommand{\sQ}{{\sf Q}}
\newcommand{\bv}{\bigvee}
	\newcommand{\bw}{\bigwedge}

\begin{document}

\title{Continuous $[0,1]$-lattices and injective
 $[0,1]$-approach spaces }

\author{Junche Yu, Dexue Zhang  \\ {\small School of Mathematics, Sichuan University, Chengdu 610064, China}\\  {\small cqyjc@icloud.com, dxzhang@scu.edu.cn}  }

 \date{}

\maketitle

\begin{abstract}In 1972, Dana Scott proved a fundamental result on the connection between order and topology which says  that   injective $T_0$ spaces are precisely continuous lattices endowed with  Scott topology. This paper investigates whether this is true in an enriched context, where the enrichment is the quantale obtained by equipping the interval $[0,1]$   with a continuous t-norm. It is shown that for each continuous t-norm, the specialization $[0,1]$-order of a separated and injective $[0,1]$-approach space $X$  is a continuous $[0,1]$-lattice and the $[0,1]$-approach structure of $X$ coincides with the Scott $[0,1]$-approach structure of its specialization $[0,1]$-order; but, unlike in the classical situation, the converse fails in general.


\end{abstract}

\section{Introduction}

The specialization order  of a topological space $X$ is the order relation $$x\leq y\quad \text{if $x$ is in the closure of $\{y\}$}.$$    It is obvious that $X$ is  a $T_0$ space if and only if the specialization order satisfies the axiom of anti-symmetry. Taking specialization defines a functor \[\Omega:{\sf Top}\lra{\sf Ord}\] from the category of topological spaces to that of ordered sets and order-preserving maps. The specialization order of any space that satisfies the separation axiom of $T_1$   is always the discrete order, so, for a general  space, the functor $\Omega$ forgets  much information of that space.

There are spaces for which $\Omega$ ``forgets nothing''. If $X$ is an Alexandroff space, that means, every point of $X$ has a smallest neighborhood, then the  structure of $X$ can be recovered from its specialization order. Actually, the functor $\Omega$ has a left adjoint \[\Gamma:{\sf Ord}\lra{\sf Top}\] that sends each ordered set $(X,\leq)$ to the space having all the lower sets of $(X,\leq)$ as closed sets.  Alexandroff spaces are then those  spaces $X$ for which $\Gamma\Omega(X)=X$.

A celebrated result of Scott \cite{Scott72}
says that the structure of every injective $T_0$ space is also encoded in its specialization order. In the words of Scott, these spaces ``are at the same time complete
lattices whose topology is determined by the lattice structure in a
special way''. That  \emph{special way} means the functor \[\Sigma:{\sf Ord}^\uparrow\lra{\sf Top}\] from the category of ordered sets and Scott continuous maps (i.e., maps that preserve directed joins) to that of topological spaces,  topologizing each ordered set with its Scott topology. Precisely, a $T_0$ space $X$ is injective   if, and only if, the specialization order of $X$ is a continuous lattice and  the topology of $X$ coincides with the Scott topology of its specialization order.  The functor $\Sigma$, which ``throws new light on the connections between  topology and lattices''\footnote{Quoted from Scott \cite[page 98]{Scott72}}, plays a prominent role in domain theory \cite{Gierz2003,Goubault}.

In this paper, we investigate the relationship between enriched orders and enriched topologies; in particular, we  address  the question whether there is a Scott-type result in the enriched context. To be precise, we investigate the relationship between ordered sets and approach spaces valued in the quantale $([0,1],\with,1)$, where $\&$ is a continuous t-norm. That means,  we take $[0,1]$-ordered sets as enriched ordered sets and take  $[0,1]$-approach spaces as enriched topological spaces. It is proved that the specialization $[0,1]$-order of an injective and separated $[0,1]$-approach space $X$ is a continuous $[0,1]$-lattice   and  the structure of $X$ is completely determined by its specialization $[0,1]$-order.  But, the converse fails in general. For a continuous t-norm $\with$  that satisfies the condition (S) (see Proposition \ref{S} below),   a necessary and sufficient condition for all continuous $[0,1]$-lattices with the Scott $[0,1]$-approach structure to be injective is that $\with$ is isomorphic to the \L ukasiewicz t-norm. This shows that in the enriched context, the Scott-type result depends on the logic structure of the truth-values. This  is an example of a not uncommon  phenomenon in the theory of fuzzy sets: to choose a quantale is to choose to work with a specific logic, the validity of some statements may depend on your choice (that is, your logic).

We end the introduction with a few words on related works.

When  the continuous t-norm is isomorphic to the product t-norm,   the quantale $([0,1],\with,1)$ is isomorphic to Lawvere's quantale $([0,\infty]^{\rm op},+,0)$,  hence  $[0,1]$-approach spaces in this case are essentially approach spaces in the sense of Lowen \cite{RL89,Lowen15}. Inspired by the fact that injective objects in the category of partially ordered sets and order-preserving maps are precisely the complete lattices, in \cite{Hof2011,Hof2013}   Hofmann  characterizes  injective approach spaces  as \emph{cocomplete spaces}.  For a comprehensive account of the story of topological spaces being viewed as categories or lax algebras, we refer to the monograph \cite{Monoidal top}.
Furthermore,  injective approach  spaces are characterized in \cite{GH}  as continuous lattices equipped with a \emph{$[0,\infty]$-action} that satisfies certain conditions. The main difference  between these works and this paper is that here the focus is on an enriched version of \emph{Scott topology}.

In \cite{Yao16}  a frame-valued version of the result of Scott is established, where  fuzzy topological spaces (valued in a frame) play the role of enriched topological spaces and flat ideals (see \cite[Remark 3.5]{LZZ2020}) play the role of directed lower sets. Compared with \cite{Yao16}, we have a different kind of quantales (in which the semigroup operation need not be idempotent), a different kind of enriched topological spaces, and a different kind of enriched directed lower sets.

Finally, this paper is closely related to  and  a continuation    of \cite{LiZ18b}. The connection and the difference are: while \cite{LiZ18b} is essentially about $[0,1]$-approach spaces with respect to the product t-norm,  this paper deals with $[0,1]$-approach spaces for all continuous t-norms.

\section{Continuous t-norms and quantale-valued orders}

In this section we recall some basic notions needed in this paper and fix some notations.

By a quantale we mean a \emph{commutative  and unital  one}  in the sense of \cite{Rosenthal1990}.  Explicitly, a quantale
\[\sQ=(\sQ,\with,k)\]
is a commutative monoid with $k$ being the unit, such that the underlying set $\sQ$ carries a complete order with a bottom element $0$ and a top element $1$,   and that the multiplication $\with$ distributes over arbitrary joins. We say that $\sQ$ is \emph{integral}  if the unit $k$ coincides with the top element   $1$. The  multiplication $\&$ determines a binary operator $\ra$,   sometimes called the \emph{implication operator} of $\&$, via the adjoint property:
\[p\with q\leq r\iff q\leq p\ra r.\]

In the language of category theory, a quantale is precisely a small, complete and symmetric monoidal closed category \cite{Borceux1994,Kelly,Lawvere1973}.

Quantales obtained by endowing the unit interval $[0,1]$ with a continuous t-norm are of particular interest in this paper.
A continuous t-norm  \cite{Klement2000} is, actually,  a continuous map $\with\colon [0,1]^2\lra[0,1]$  that makes $([0,1],\with ,1)$ into a quantale. Such quantales play a decisive role in the BL-logic of H\'{a}jek \cite{Ha98}.

Basic   continuous t-norms and their implication operators are listed below:
\begin{enumerate}[label={\rm(\roman*)}] \setlength{\itemsep}{0pt}
\item  The G\"{o}del t-norm: \[ x\with y= \min\{x,y\}; \quad x\ra y=\begin{cases}
		1,&x\leq y,\\
		y,&x>y.
		\end{cases} \] The implication  operator $\ra$ of the G\"{o}del t-norm is  continuous  except at   $(x,x)$, $x<1$.

\item  The product t-norm:  \[ x\with_P y=xy; \quad x\ra y=\begin{cases}
		1,&x\leq y,\\
		y/x,&x>y.
		\end{cases} \] The implication operator  $\ra$
of the product t-norm is  continuous  except at   $(0,0)$. The quantale $([0,1], \with_P, 1)$ is  isomorphic to Lawvere's quantale $([0,\infty]^{\rm op},+,0)$ \cite{Lawvere1973}.

\item  The {\L}ukasiewicz t-norm:\[ x\with_{\L} y=\max\{0,x+y-1\}; \quad x\ra y=\min\{1-x+y,1\}. \] The implication operator  $\ra$
    of the {\L}ukasiewicz t-norm
    is   continuous on $[0,1]^2$.
	\end{enumerate}

Let $\with $ be a continuous t-norm. An element $p\in [0,1]$ is   \emph{idempotent}  if $p\with p=p$.

\begin{prop}{\rm(\cite[Proposition 2.3]{Klement2000})} \label{idempotent}
Let $\&$ be a continuous t-norm on $[0,1]$ and $p$ be an idempotent element of $\&$. Then $x\with y= \min\{x, y\} $ whenever $x\leq p\leq y$.
 \end{prop} It follows immediately  that $y\ra x=x$ whenever  $x< p\leq y$ for some idempotent   $p$. Another consequence of  Proposition \ref{idempotent} is that for any idempotent elements $p, q$    with $p<q$,  the restriction of $\with $ to $[p,q]$, which is also denoted by $\with$,  makes $[p,q]$ into a commutative quantale with $q$ being the unit element.    The following  theorem, known as the \emph{ordinal sum decomposition theorem},  is of fundamental importance.

\begin{thm} {\rm(\cite{Klement2000,Mostert1957})}
\label{ordinal sum} Let $\with $ be a continuous t-norm. If $a\in [0,1]$ is non-idempotent, then there exist idempotent elements $a^{-}, a^{+}\in [0,1]$ such that $a^-<a<a^+$ and   the quantale  $([a^{-},a^{+}],\with ,a^{+})$ is either isomorphic to  $([0,1],\with_{\L},1)$ or to $([0,1],\with_{P},1)$. \end{thm}


\begin{prop}{\rm(\cite{LZ2020}\label{S})}
For a continuous t-norm $\with$ on $[0,1]$, the following statements are equivalent:
\begin{enumerate}[label=\rm(\arabic*)] \setlength{\itemsep}{0pt}
	\item  The implication $\ra\colon [0,1]\times [0,1]\lra [0,1]$ is continuous at each point off the diagonal $\{(x,x)\mid x\in[0,1]\}$.
	\item  For each non-idempotent element $a\in[0,1]$, the quantale $([a^-,a^+],\with, a^+)$ is isomorphic to $([0,1],\with_P,1)$ whenever $a^->0$.
\end{enumerate}
\end{prop}

We say that a continuous t-norm $\with $ on $[0,1]$ satisfies the \emph{condition (S)} if it satisfies   the equivalent conditions in Proposition \ref{S}.

A \emph{$\sQ$-valued order} (a $\sQ$-order for short) is a map $\alpha\colon X\times X\lra \sQ$ such that for all $x,y,z\in X$,
$$\alpha(x,x)\geq k\quad \text{and}\quad  \alpha(y,z)\with \alpha(x,y)\leq\alpha(x,z).$$
The pair $(X,\alpha)$  is call a \emph{$\sQ$-ordered set} (or, a $\sQ$-category in the language of enriched categories).

If $\alpha$ is a $\sQ$-order on $X$, it follows from the commutativity of $\with$ that $\alpha^{\rm op}(x,y)\coloneqq\alpha(y,x)$ is also a $\sQ$-order on $X$, called the opposite of $\alpha$.

It is customary to write $X$ for the pair $(X,\alpha)$ and write $X(x,y)$ for $\alpha(x,y)$.
We say that a map $f\colon X\lra Y$ between $\sQ$-ordered sets \emph{preserves $\sQ$-order}   if $X(x,y)\leq Y(f(x),f(y))$ for all $x,y\in X$.  $\sQ$-ordered sets and $\sQ$-order-preserving maps constitute a (locally ordered) category
 \[\sQ\text{-}\sf Ord.\]

\begin{rem} In 1973, Lawvere   \cite{Lawvere1973} argued that any closed category, a quantale in particular, can be viewed as the table of truth-values of a generalized logic, hence the theory of enriched categories is of a logic nature. This idea has led to the study of quantale-enriched categories, generalized metric spaces in particular,   as \emph{quantitative domains}. We list a few of the early works here: \cite{AR89,BvBR1998,Flagg1997,FSW,Wagner97}. \end{rem}

The \emph{underlying order} of a $\sQ$-ordered set $X$ refers to the order $\leq$ given by $$x\leq   y \quad {\rm if}~ k\leq X(x,y) .$$ For convenience, for each $\sQ$-ordered set $X$, we shall write $X_0$ for the ordered set obtained by equipping $X$ with the underlying order.

Two elements $x,y$ of a $\sQ$-ordered set $X$ are \emph{isomorphic} if \(k\leq X(x,y)\wedge X(y,x).\) A $\sQ$-ordered set $X$ is \emph{separated} if isomorphic elements are identical; that is, \[k\leq X(x,y)\wedge X(y,x) \quad \Rightarrow\quad x=y.\]
Said differently, $X$ is separated if $X_0$ satisfies the axiom of anti-symmetry.

\begin{exmp}
For all $x,y\in \sQ$, let
$$\alpha_L(x,y)=x\ra y.$$ Then $\alpha_L$ is a separated $\sQ$-order on $\sQ$. We denote the opposite of $\alpha_L$ by $\alpha_R$; that is, $$  \alpha_R(x,y)=y\ra x.$$
Both $(\sQ,\alpha_L)$ and $(\sQ,\alpha_R)$ play important roles in this paper.
\end{exmp}
\begin{exmp}For each set $X$, the map \[\sub_X\colon \sQ^X\times\sQ^X\lra\sQ, \quad  \sub_X(\phi,\psi)=\bigwedge_{x\in X}\phi(x)\ra\psi(x)\] is a  $\sQ$-order on the set $\sQ^X$, known as the fuzzy inclusion $\sQ$-order.   If $X$ is a singleton set, then $(\sQ^X,\sub_X)$ degenerates to the $\sQ$-ordered set $(\sQ,\alpha_L)$.
\end{exmp}

Taking underlying order defines a functor from $\sQ\text{-}\sf Ord$ to the  category of ordered sets and order-preserving maps $$\iota\colon  \sf \sQ\text{-}Ord\lra Ord.$$
The functor $\iota$ has a left adjoint $$\omega\colon  \sf Ord\lra\sQ\text{-}Ord $$
that sends an order $\leq$ on a set $X$ to the $\sQ$-order   $\alpha$ on $X$ given by   $\alpha(x,y)=k$ if $x\leq y$, and $\alpha(x,y)=0$  if $x\not\leq y$.

The category $\sQ\text{-}\sf Ord$ is complete and cocomplete; for example, the product $\prod_iX_i$ of a family of $\sQ$-ordered sets $\{X_i\}_i$ is given by the Cartesian product of the sets $\{X_i\}_i$ and
$$\prod_iX_i(\vec{x},\vec{y})=\bigwedge_iX_i(x_i,y_i) $$ for all $\vec{x},\vec{y}\in\prod_iX_i$.
It is clear that for each set $X$,   $(\sQ^X,\sub_X)$ is a product of the $\sQ$-ordered set $(\sQ,\alpha_L)$.

Let $X,Y$ be $\sQ$-ordered sets; let $f\colon X\lra Y$ and $g\colon Y\lra X$ be  maps.  We say that $f$ is left adjoint to $g$ (or, $g$ is right adjoint to $f$) and write $f\dashv g$, if
$$Y(f(x),y)=X(x,g(y))$$
for all $x\in X$ and $y\in Y$.
This is a special case of enriched adjunctions in  enriched category theory \cite{Borceux1994,Kelly}.
\begin{thm}\label{Characterization of adjoints} {\rm(\cite[page 295]{Stubbe2006})} Let $f\colon X\lra Y$ and $g\colon Y\lra X$ be a pair of maps between $\sQ$-ordered sets. Then $f$ is left adjoint to $g$ if and only if the following conditions are satisfied: \begin{enumerate}[label={\rm(\roman*)}] \setlength{\itemsep}{0pt}
  \item Both $f$ and $g$ preserve $\sQ$-order; \item The map $f\colon X_0\lra Y_0$ is left adjoint to $g\colon Y_0\lra X_0$; that is,  for all $x\in X$ and $y\in Y$, $f(x)\leq y\iff x\leq g(y)$. \end{enumerate}\end{thm}

Suppose that $X$ is a $\sQ$-ordered set. A  \emph{weight}  of  $X$, also known as a \emph{lower  fuzzy set} of  $X$, is a map $\phi\colon  X\lra \sQ$ such that for all $x,y\in X$,
$$\phi(y)\with X(x,y)\leq\phi(x).$$
In other words, a weight  of $X$ is a $\sQ$-order-preserving map $\phi\colon X\lra(\sQ,\alpha_R)$.  The weights of $X$ constitute a $\sQ$-ordered set $\mathcal{P}X$ with
$$\mathcal{P}X(\phi_1,\phi_2)\coloneqq\sub_X(\phi_1,\phi_2).$$

For each $x\in X$, $X(-,x)$ is a weight of $X$   and we have the following:

\begin{lem}[Yoneda lemma]
Let $X$ be a $\sQ$-ordered set and $\phi$ be a weight of $X$, then
$$\mathcal{P}X(X(-,x),\phi)=\phi(x).$$
\end{lem}

The above lemma is a special case of the Yoneda lemma in enriched category theory, see e.g. \cite{Borceux1994,Stubbe2005}. The Yoneda lemma entails that the map
$$\sy\colon X\lra\mathcal{P}X, \quad x\mapsto X(-,x)$$  is    an embedding if $X$ is separated. By abuse of language,  we call it the \emph{Yoneda embedding} no matter $X$ is separated or not.

Each $\sQ$-order-preserving map $f\colon X\lra Y$ gives rise to an adjunction $f^\ra\dashv f^\leftarrow$ between $\CP X$ and $\CP Y$, where  \[f^\ra\colon \mathcal{P}X\lra\mathcal{P}Y,\quad f^\ra(\phi)(y)=\bigvee_{x\in X} \phi(x)\with Y(y,f(x)) \]
and \[f^\leftarrow\colon\mathcal{P}Y\lra\mathcal{P}X, \quad f^\leftarrow(\psi)(x)=\psi(f(x)).\]

Let  $a$ be an element and $\phi$ be a weight of a $\sQ$-ordered set $X$. We say that $a$ is a  \emph{supremum}  of $\phi$ if
$$X(a,y)=\CP X(\phi, X(-,y))$$
for all $y\in X$. Suprema of a weight, if exist, are unique up to isomorphism. We say that  $X$ is \emph{cocomplete} if every weight of $X$ has a supremum. It is clear that $X$ is cocomplete if and only if the Yoneda embedding $\sy\colon X\lra\mathcal{P}X$ has a left adjoint (see e.g. \cite{Stubbe2005}).

\begin{exmp}(Examples 2.11 and   2.12 in \cite{LZZ2020}) Let $X$ be $\sQ$-ordered set. \begin{enumerate}[label={\rm(\roman*)}] \setlength{\itemsep}{0pt}
  \item If $\psi\colon X\lra (\sQ,\alpha_L)$ is a $\sQ$-order-preserving map, then for each weight $\lambda$ of $X$, \begin{equation}\label{composition as sup} \sup\psi^\ra(\lambda)=\bv_{x\in X}\lambda(x)\with\psi(x).\end{equation} \item Let $\phi$ be a weight of $X$, viewed as a $\sQ$-order-preserving map $X\lra (\sQ,\alpha_R)$. Then for each weight $\lambda$ of $X$, \begin{equation}\label{inclusion as sup} \sup\phi^\ra(\lambda) = \sub_X(\lambda,\phi).\end{equation} \end{enumerate} \end{exmp}

We say that a $\sQ$-ordered set $X$ is \emph{complete} if $X^{\rm op}$ is cocomplete. It is known that a $\sQ$-ordered set  is cocomplete if and only if it is complete and that a $\sQ$-order-preserving map $f\colon X\lra Y$ between cocomplete $\sQ$-ordered sets is a left adjoint if and only if $f$ \emph{preserves suprema} in the sense that \(f(\sup\phi)=\sup f^\ra(\phi)\) for each weight $\phi$ of $X$ (see e.g. \cite{Stubbe2005}).

A $\sQ$-ordered set is called a \emph{complete $\sQ$-lattice} if it is separated and  complete (or equivalently, cocomplete). The category of complete $\sQ$-lattices and left adjoints is denoted by \[\QSup.\]

For each $\sQ$-ordered set $X$, $\CP X$ is  a complete $\sQ$-lattice. The  left adjoint of the Yoneda embedding $\sy_{\CP X}\colon\CP X\lra\CP\CP X$ is given by $\sy_X^\leftarrow\colon \CP\CP X\lra\CP X$. A direct calculation shows that for each $\Phi\in\CP\CP X$, \[ \sy_X^\leftarrow(\Phi)= \bv_{\phi\in\CP X}\Phi(\phi)\with\phi.\]
The assignment \[f\colon X\lra Y\quad \mapsto\quad f^\ra\colon\CP X\lra\CP Y\] defines a functor $\CP\colon\QOrd\lra\QSup$ that is left adjoint to the forgetful functor $U\colon\QSup\lra\QOrd$. The monad on the locally ordered category $\sQ$-{\sf Ord} arising from the adjunction $\CP\dashv U$ is denoted by \[\mathbb{P}=(\CP,\sfm,\sy).\]     Explicitly, for a $\sQ$-ordered set $X$, \begin{itemize} \setlength{\itemsep}{0pt}\item $\CP X$ is the complete $\sQ$-lattice of all weights of $X$; \item the unit is the Yoneda embedding $\sy_X\colon X\lra\CP X$; \item the multiplication  \(\sfm_X=\sup_{\CP X}=\sy_X^\leftarrow\colon\CP\CP X\lra \CP X\).
\end{itemize}

The monad $\mathbb{P}=(\CP,\sfm,\sy)$ is of Kock-Z\"{o}berlein type and the category of $\mathbb{P}$-algebras and $\mathbb{P}$-homomorphisms is precisely the category of   $\sQ$-complete lattices and left adjoints  (see e.g. \cite{Stubbe2005}); that is,
\[\QSup=\mathbb{P}\text{-}{\sf Alg}.\]
In particular, the forgetful functor $\QSup\lra\QOrd$ is monadic.

The following conclusion is an instance of \cite[Proposition 3.1]{LZ2020}, which is quite likely to have appeared somewhere else.

\begin{prop} \label{retractcomplete}
Every retract of a complete $\sQ$-lattice in $\sQ\text{-}\sf Ord$ is a complete $\sQ$-lattice.
\end{prop}
As in the classical situation, complete $\sQ$-lattices are precisely  the separated and injective objects in   $\QOrd$.

\section{Quantale-valued approach spaces}

While a $\sQ$-ordered set  is   an ordered set enriched over $\sQ$, a $\sQ$-approach space can be thought of as a   topological space enriched over $\sQ$.

\begin{defn} Let $\sQ=(\sQ,\with,k)$ be a quantale.
A $\sQ$-valued approach structure on a  set $X$ is a map $\delta\colon X\times 2^X\lra \sQ$ subject to the  following conditions: for all $x\in X$, and $A,B\in 2^X$,
\begin{itemize} \setlength{\itemsep}{0pt}
	\item[(A1)] $\delta(x,\{x\})\geq k$;
	\item[(A2)] $\delta(x,\varnothing)=0$;
	\item[(A3)] $\delta(x,A\cup B)= \delta(x,A)\vee\delta(x,B)$;
	\item[(A4)] $\delta(x,A)\geq  \Big(\bw\limits_{b\in B}\delta(b,A)\Big)\with\delta(x,B)$.
\end{itemize}
The pair $(X,\delta)$  is called a $\sQ$-valued approach space,   a $\sQ$-approach space for simplicity.
\end{defn}

\begin{rem}In a $\sQ$-approach space, the value $\delta(x,A)$ measures how close a point is to a subset. When $\sQ$ is the boolean algebra $(\{0,1\},\wedge,1)$, a $\sQ$-approach space is precisely a topological space with $\delta(x,A)$ denoting whether $x$ is in the closure of $A$.  So,  $\sQ$-approach spaces can be viewed as ``quantitative topological spaces''; actually, they are called \emph{$\sQ$-valued topological spaces} in \cite{LT17a,LT17b}. $\sQ$-approach spaces have appeared in the literature under different names. In \cite{FH90,Ho78}, $\sQ$-approach structures are called  \emph{fuzzy contiguity relations}. If $\sQ$ is the unit interval equipped with a left continuous t-norm,  such spaces  are  \emph{probabilistic topological spaces} in the sense of  Frank \cite{Frank} with some extra requirements. In the case that $\sQ$ is Lawvere's quantale $([0,\infty]^{\rm op},+,0)$, such spaces are introduced in Lowen \cite{RL89} under the name  \emph{approach spaces}. We agree by convention that for  Lawvere's quantale $\sQ=([0,\infty]^{\rm op},+,0)$, we  speak of approach spaces, instead of $\sQ$-approach spaces or $[0,\infty]$-approach spaces. \end{rem}

A map $f\colon (X,\delta_X)\lra (Y,\delta_Y)$ between $\sQ$-approach spaces is \emph{continuous} if for all $x\in X$ and $A\subseteq X$,
$$\delta_X(x,A)\leq\delta_Y(f(x),f(A)).$$
$\sQ$-approach spaces and continuous maps constitute a category
$$\sQ\sf\text{-} App.$$   
It is shown in \cite{LT17b} that if the underlying lattice of $\sQ$ has enough coprimes, then $\sQ\sf\text{-} App$ is  topological  over the category of sets in the sense of \cite{AHS}; in particular, it is complete and cocomplete.

\begin{exmp} \label{the space K} Let $\sQ$ be a linearly ordered quantale. Then the map \[\delta_{\mathbb{K}}\colon\sQ\times2^\sQ\lra\sQ, \quad
\delta_{\mathbb{K}}(x,A)=
\begin{cases}
\inf A \rightarrow x, & A\neq\varnothing,\\
0, & A=\varnothing
\end{cases}\]
is a $\sQ$-approach structure on $\sQ$. The space $$\mathbb{K}\coloneqq(\sQ,\delta_{\mathbb{K}})$$  plays an important role in this paper. When $\sQ$ is Lawvere's quantale $([0,\infty]^{\rm op}, +,0)$, this space is precisely the approach space $\mathbb{P}$ introduced by Lowen (see e.g. \cite[Examples 1.8.33\thinspace(2)]{RL97}).
\end{exmp}

\begin{prop}The space  $\mathbb{K}$ is an initially dense object in the category $\sQ\text{-}\sf App$.
\end{prop}

\begin{proof}For each $\sQ$-approach space $(X,\delta)$ and each subset $A$ of $X$, it follows   from (A4) that $\delta(-,A):(X,\delta)\lra (\sQ,\delta_\mathbb{K})$ is a continuous map. Then, it is readily seen that \[\{(X,\delta)\to^{\delta(-,A)} (\sQ,\delta_\mathbb{K})\}_{A\subseteq X}\] is an initial source.\end{proof}

As one might expect, the relation between $\sQ$-ordered sets and $\sQ$-approach spaces is analogous to that between ordered sets and topological spaces.
For each $\sQ$-approach space $(X,\delta)$, the $\sQ$-relation \[\Omega(\delta)\colon X\times X\lra\sQ, \quad \Omega(\delta)(x,y)=\delta(x,\{y\})\]   is a $\sQ$-order on $X$, called the \emph{specialization $\sQ$-order} of $(X,\delta)$.
Assigning to each $\sQ$-approach space its  specialization $\sQ$-order gives rise to a functor $$\Omega\colon\sQ\text{-}{\sf App}\lra\sQ\text{-}{\sf Ord}.$$  This functor has a left adjoint $$\Gamma\colon\sQ\text{-}{\sf Ord}\lra\sQ\text{-}{\sf App}$$ that sends each $\sQ$-ordered set $(X,\alpha)$ to the $\sQ$-approach space $(X,\Gamma(\alpha))$ given by
\[
\Gamma(\alpha)(x,A)=\bv\limits_{a\in A} \alpha(x,a).
\]
The functor $\Gamma$ is   full and faithful.

The adjunction $\Gamma\dashv\Omega$ is  an analogy in the $\sQ$-valued context of that between the categories of ordered sets and topological spaces mentioned in the introduction. In the classic case, the left adjoint  sends every ordered set to its Alexandroff topology and the right adjoint sends each topological space to its specialization order.

Now we postulate a class of $\sQ$-approach spaces making use of the specialization $\sQ$-order. These spaces are an extension of $T_0$ topological spaces.

\begin{defn}A $\sQ$-approach space $(X,\delta)$  is  separated if its specialization $\sQ$-order $\Omega(\delta)$  is separated.\end{defn}

Given a $\sQ$-approach space $(X,\delta)$ and a subset $Y$ of $X$,   restricting the domain of $\delta$ to $Y\times 2^Y$ yields a $\sQ$-approach structure on $Y$, the resulting space is called a subspace of $(X,\delta)$.

A $\sQ$-approach space $(Z,\delta_Z)$ is   \emph{injective} if for each $\sQ$-approach space $(X,\delta_X)$ and each continuous map $f$ from a subspace $(Y,\delta_Y)$ of $(X,\delta_X)$ to $(Z,\delta_Z)$, there is a continuous map $\overline{f}:(X,\delta_X)\lra(Z,\delta_Z)$ that extends $f$. 
The main question of this paper is:

\begin{ques}
Are the separated and injective $\sQ$-approach spaces precisely continuous lattices enriched over $\sQ$ as in the classical situation?  \end{ques}

As we shall see in the case that $\sQ$ is the interval $[0,1]$ together with a continuous t-norm, the answer depends on the structure of the continuous t-norm which plays the role of the connective \emph{conjunction} in fuzzy logic.

\section{Relation to $\sQ$-cotopological spaces}

Approach spaces (i.e., $\sQ$-approach spaces when $\sQ=([0,\infty]^{\rm op},+,0)$) have many equivalent descriptions \cite{RL97,Lowen15}, one of them is the description by \emph{lower regular functions}. In this section, we show that there is a similar description of $\sQ$-approach spaces in the circumstance that $\sQ$ is the interval $[0,1]$ endowed with a continuous t-norm; precisely, for such a quantale the category of $\sQ$-approach spaces is isomorphic to that of strong $\sQ$-cotopological spaces.

A notation first.
Let $X$ be a set and let $A$ be a subset of $X$. For each $p\in\sQ$, we write $p_A$ for the map $X\lra\sQ$ given by \[ p_A(x)= \begin{cases} p, & x\in A,\\ 0, & x\notin A. \end{cases} \]

By a \emph{$\sQ$-cotopology}  on a set $X$ we mean a subset $\tau$ of $\sQ^X$ subject to the following conditions:
\begin{enumerate}[label=\rm(C\arabic*)] \setlength{\itemsep}{0pt}
\item  $0_X\in\tau$;
\item  $\phi\vee\psi\in\tau$ for all $\phi, \psi\in\tau$;
	\item  $\bigwedge_{i}\phi_i\in\tau$ for each subset $\{\phi_i\}$ of $\tau$; \item $p\rightarrow\phi\in\tau$ for all $p\in\sQ$ and $\phi\in\tau$.\end{enumerate} The pair $(X,\tau)$ is called a  $\sQ$-cotopological space; elements of $\tau$ are called closed sets.

Let $(X,\tau)$ be a $\sQ$-cotopological space. The closure operator of $(X,\tau)$ is the map
$(-)^-\colon\sQ^X\lra\sQ^X$ given by
\[\overline{\mu}=\bw\{\phi\in\tau\mid\mu\leq\phi\}.\]  The closure operator  satisfies the following conditions: for all $\lambda,\mu\in\sQ^X$,
\begin{enumerate}[label=\rm(cl\arabic*)] \setlength{\itemsep}{0pt}
\item  $\overline{0_X}=0_X$;
\item  $ \overline{\mu} \geq \mu$;
\item
$\overline{\lambda\vee\mu}=\overline{\lambda}\vee\overline{\mu}$;
\item $\overline{\overline{\mu}}=\overline{\mu}$; \item \(\sub_X(\lambda,\mu)\leq\sub_X(\overline{\lambda},\overline{\mu})\) for all $\lambda,\mu\in\sQ^X$.
\end{enumerate}
Actually,  $\sQ$-cotopologies on $X$ correspond bijectively to operators $\sQ^X\lra\sQ^X$ satisfying (cl1)-(cl5).





A map $f\colon (X,\tau_X)\lra (Y,\tau_Y)$ between $\sQ$-cotopological spaces is continuous if $\lambda\circ f\in\tau_X$ for all $\lambda\in\tau_Y$. $\sQ$-cotopological spaces and continuous maps constitute a topological category
$$\sf \sQ\text{-}CTop.$$

Given a $\sQ$-cotopological space $(X,\tau)$, the map \[\zeta(\tau)\colon X\times 2^X\lra\sQ, \quad \zeta(\tau)(x,A)=\overline{k_A}(x)\] is a  $\sQ$-approach structure on $X$. The assignment $(X,\tau)\mapsto(X,\zeta(\tau))$ gives rise to a functor \[\zeta\colon\sQ\text{-}{\sf CTop}\lra\sQ\text{-}{\sf App}.\]

The \emph{specialization $\sQ$-order} $\Omega(\tau)$ of a $\sQ$-cotopological space $(X,\tau)$ is defined to be the specialization $\sQ$-order of the $\sQ$-approach space $\zeta(X,\tau)$; that is,  $\Omega(\tau)(x,y)= \overline{k_y}(x)$. We say that $(X,\tau)$ is a $T_0$ space if $\Omega(\tau)$ is separated, or equivalently, if $\zeta(X,\tau)$ is separated.

A $\sQ$-cotopology on a set $X$ is essentially a map $$\tau\colon X\times \sQ^X\lra\sQ$$ subject to the following conditions: \begin{enumerate}[label=(\roman*)] \setlength{\itemsep}{0pt}
	\item  $\tau(x,\lambda)\geq \lambda(x)$;
	\item  $\tau(x,0_X)=0$;
	\item  $\tau(x,\lambda\vee\mu)= \tau(x,\lambda)\vee\tau(x,\mu)$;
	\item  $\tau(x,\lambda)\geq\tau(x,\mu)\with \sub_X(\mu,\tau(-,\lambda))$.
\end{enumerate}
The $\sQ$-approach space  $\zeta(X,\tau)$  is obtained by restricting the domain of   $\tau\colon X\times \sQ^X\lra\sQ$ to $X\times 2^X$. Thus, the  functors \[\sQ\text{-}\sf CTop\to^\zeta \sQ\text{-}App \to^\Omega\sQ\text{-}Ord\] are obtained by composing the corresponding structure maps  respectively with the embeddings  \[X\times X\to^{(x,y)\mapsto(x,\{y\})}X\times 2^X\to^{(x,A)\mapsto(x,k_A)} X\times\sQ^X.\]  Conversely, to make a $\sQ$-approach space into a $\sQ$-cotopological space, we need to   extend  the domain of $\delta\colon X\times 2^X\lra \sQ$ to $X\times \sQ^X$. It seems hard to do so for a general quantale, however, there is a natural way to do so  when  $\sQ$ is   linearly ordered, as we now see.

\begin{lem}Let $\sQ$ be a linearly ordered quantale. Then for each $\sQ$-approach space $(X,\delta)$, the set $\kappa(\delta)$ of all continuous maps $(X,\delta)\lra \mathbb{K}$ is a $\sQ$-cotopology on $X$. \end{lem}

\begin{proof} We leave it to the reader to check that  $\kappa(\delta)$ satisfies the conditions (C1) and (C3). Here we check that  it satisfies the conditions (C2) and (C4).

Let $\lambda\colon (X,\delta)\lra \mathbb{K}$ be a continuous map and let $p\in\sQ$. Then for each $x\in X$ and each $A\subseteq X$, if $A$ is empty, it holds trivially that \[\delta(x,A)\leq\delta_\mathbb{K}((p\ra\lambda)(x),(p\ra\lambda)(A));\]  if $A\not=\varnothing$, then \begin{align*}\delta(x,A) &\leq \inf\lambda(A) \ra\lambda(x)\\ &\leq (p\ra\inf\lambda(A))\ra(p\ra\lambda(x))\\ &= \delta_\mathbb{K}((p\ra\lambda)(x),(p\ra\lambda)(A)).\end{align*} Thus, $p\ra\lambda\colon (X,\delta)\lra \mathbb{K}$ is a continuous map and $\kappa(\delta)$ satisfies (C4).

To show that $\kappa(\delta)$ satisfies (C2), we make use of the  assumption that $\sQ$ is linearly ordered. Assume that $\lambda$ and $\mu$ are  continuous maps from $(X,\delta)$ to $(\sQ,\delta_\mathbb{K})$. Let $x\in X$ and $A\subseteq X$. We need to show that \[\delta(x,A)\leq \delta_\mathbb{K}(\lambda(x)\vee\mu(x),(\lambda \vee\mu)(A)).\] Let \[B=\{a\in A\mid \mu(a)\leq\lambda(a)\}; \quad C= \{a\in A\mid \lambda(a)\leq\mu(a)\}.\] Since $\sQ$ is linearly ordered, it follows that $A=B\cup C$, thus, \begin{align*}\delta(x,A)&= \delta(x,B)\vee\delta(x,C) \\ &\leq \delta_\mathbb{K}(\lambda(x), \lambda (B))\vee \delta_\mathbb{K}(\mu(x), \mu (C)) \\  &\leq  \delta_\mathbb{K}(\lambda(x)\vee\mu(x), \lambda (B))\vee \delta_\mathbb{K}(\lambda(x)\vee\mu(x), \mu (C)) \\ &= \delta_\mathbb{K}(\lambda(x)\vee\mu(x),(\lambda \vee\mu)(A)), \end{align*}which completes the proof. \end{proof}

If $\sQ$ is Lawvere's quantale $([0,\infty]^{\rm op},+,0)$, then  for each $\sQ$-approach space $(X,\delta)$, the set  $\kappa(\delta)$ is precisely that of lower regular functions of $(X,\delta)$ (see e.g. \cite{RL97,Lowen15}).

\begin{exmp}Let $\sQ$ be a linearly ordered quantale; let  $(X,\alpha)$ be a $\sQ$-ordered set. Since for each map $\lambda\colon X\lra\sQ$, we have \begin{align*}&\quad\quad\quad \lambda\colon  (X,\Gamma(\alpha))\lra \mathbb{K}~\text{is continuous}\\ &\iff \forall x\in X,\forall A\subseteq X, \Gamma(\alpha)(x,A)\leq \delta_\mathbb{K}(\lambda(x),\lambda(A))\\ &\iff \forall x\in X,\forall A\not=\varnothing, {\textstyle\bv_{a\in A}X(x,a)\leq \inf\lambda(A)}\ra\lambda(x)\\   &\iff \forall x,a\in X, X(x,a)\leq \lambda(a)\ra\lambda(x)\\ &\iff \lambda\colon X\lra(\sQ,\alpha_R)~\text{preserves $\sQ$-order}, \end{align*} it follows that the closed sets of  $\kappa\circ\Gamma(\alpha)$ are exactly the weights of $(X,\alpha)$.
\end{exmp}

\begin{exmp}\label{cotopology of K} Let $\sQ=([0,1],\with,1)$ with $\with$ being a  left continuous t-norm. Then a map $\lambda\colon [0,1]\lra[0,1]$ is a closed set of the $\sQ$-cotopological space $([0,1],\kappa(\delta_\mathbb{K}))$ if and only if  it satisfies: \begin{enumerate}[label=(\roman*)] \setlength{\itemsep}{0pt} \item $y\ra x\leq\lambda(y)\ra\lambda(x)$ for all $x,y\in[0,1]$;   \item $\lambda$ is right continuous. \end{enumerate}

By definition,  we have \begin{align*}\lambda\in \kappa(\delta_\mathbb{K})&\iff \lambda\colon  ([0,1],\delta_\mathbb{K})\lra ([0,1],\delta_\mathbb{K})~\text{is continuous}\\ &\iff \forall x\in X,\forall A\not=\varnothing, \delta_\mathbb{K}(x,A)\leq \delta_\mathbb{K}(\lambda(x),\lambda(A))\\ &\iff \forall x\in X,\forall A\not=\varnothing, \inf A\ra x\leq \inf\lambda(A) \ra\lambda(x). \end{align*}

Suppose that $\lambda\in \kappa(\delta_\mathbb{K})$.  Letting $A=\{y\}$ gives that  $y\ra x\leq\lambda(y)\ra\lambda(x)$, in particular, $\lambda$ is non-decreasing. Letting $x=\inf A$ for a nonempty subset gives that $\lambda(\inf A)=\inf\lambda(A)$, so $\lambda$ is right continuous. This shows every closed set of  $([0,1],\kappa(\delta_\mathbb{K}))$ satisfies (i) and (ii).
The converse implication is clear.
\end{exmp}

\begin{prop}Let $\sQ$ be a linearly ordered quantale. Then, the assignment $(X,\delta)\mapsto (X,\kappa(\delta))$ gives rise to a functor
\[\kappa\colon\sQ\text{-}{\sf App}\lra\sQ\text{-}{\sf CTop}.\] Furthermore, $\kappa$ is  left adjoint and right inverse to \[\zeta\colon\sQ\text{-}{\sf CTop}\lra\sQ\text{-}{\sf App}.\] \end{prop}

\begin{proof} First, we prove a useful fact: for each $\sQ$-approach space $(X,\delta)$ and each subset $A$ of $X$, the map \[\delta(-,A)\colon X\lra\sQ, \quad x\mapsto \delta(x,A)\] is a closed set of $(X,\kappa(\delta))$ and it is the closure   of $k_A$ in $(X,\kappa(\delta))$; that is,  $\overline{k_A}=\delta(-,A)$.

On the one hand, it follows   from (A4) that $\delta(-,A):(X,\delta)\lra (\sQ,\delta_\mathbb{K})$ is a continuous map, hence a closed set of  $(X,\kappa(\delta))$. Thus,   $\overline{k_A}\leq\delta(-,A)$. On the other hand, if $\lambda$ is a closed set of  $(X,\kappa(\delta))$ satisfying $k_A\leq\lambda$, then for all $x\in X$, \[\delta(x,A)\leq \delta_\mathbb{K}(\lambda(x), \lambda(A)) = \inf\lambda(A) \ra\lambda(x)\leq \lambda(x),\] hence $\delta(-,A)\leq \lambda$. Therefore, $\overline{k_A}=\delta(-,A)$.

From this fact one immediately infers that for each $\sQ$-approach space $(X,\delta)$, $\zeta\circ\kappa(\delta)=\delta$, hence $\kappa$ is right inverse to $\zeta$. To see that $\kappa$ is left adjoint to $\zeta$, we show that for each $\sQ$-cotopological space $(Y,\tau)$, if $f\colon (X,\delta)\lra(Y,\zeta(\tau))$ is a continuous map  between $\sQ$-approach spaces, then $f\colon (X,\kappa(\delta))\lra(Y, \tau )$ is a continuous map  between $\sQ$-cotopological spaces. That means, for each  closed set $\lambda$ of $(Y,\tau)$,  the map  $\lambda\circ f\colon (X,\delta)\lra (\sQ,\delta_\mathbb{K})$ is   continuous. Let $x$ be an element   and   $A$ be a subset of $X$. Without loss of generality, we assume that $A\not=\varnothing$. Let $p$ be the meet of   \( \lambda\circ f (A)\) in $\sQ$. Then $p\ra\lambda$ is a closed set of $(Y,\tau)$ with $(p\ra\lambda)(f(a))\geq k$ for all $a\in A$, so the closure of $k_{f(A)}$ in $(Y,\tau)$ is contained in $p\ra\lambda$, i.e.,   $\overline{k_{f(A)}}\leq p\ra\lambda$. Therefore, \begin{align*}\delta(x,A)&\leq \zeta(\tau)(f(x),f(A))\\ &= \overline{k_{f(A)}}(f(x))\\ &\leq (p\ra\lambda)(f(x))\\ &= \inf \lambda\circ f (A)\ra \lambda\circ f(x)\\ &=\delta_\mathbb{K}(\lambda\circ f(x), \lambda\circ f (A)),  \end{align*} which shows that $\lambda\circ f\colon (X,\delta)\lra (\sQ,\delta_\mathbb{K})$ is   continuous, as desired. \end{proof}

If the quantale $\sQ$ is the interval $[0,1]$ endowed with a continuous t-norm, there is a nice characterization of the image of the functor $\kappa$, it consists precisely of  the strong $\sQ$-cotopological spaces.
A $\sQ$-cotopology  $\tau$ on a set $X$ is  \emph{strong} \cite{Zhang2007} if it satisfies
\begin{enumerate}
\item[(C5)] $p\with \phi\in\tau$ for all $p\in \sQ$ and $\phi\in\tau$.
\end{enumerate}

\begin{rem}If $\sQ$ is Lawvere's quantale $\sQ=([0,\infty]^{\rm op},+,0)$, then a strong $\sQ$-cotopology on a set is exactly a \emph{lower regular function frame} in the sense of Lowen \cite{RL97,Lowen15}.  \end{rem}

\begin{prop} {\rm (\cite{CLZ2011})} \label{closure:cl6} Let $(X,\tau)$ be a  $\sQ$-cotopological space. Then $(X,\tau)$ is  strong   if and only if its closure operator  satisfies
\begin{enumerate}
\item[\rm(cl6)]   \(\overline{p\with \mu}= p\with \overline{\mu}\) for all $p\in \sQ$ and $\mu\in\sQ^X$. \end{enumerate}
\end{prop}

If $\sQ$ is the interval $[0,1]$ endowed with a  continuous t-norm,  strong  $\sQ$-cotopological spaces  are, in essence,  fuzzy $T$-neighborhood spaces in \cite{Morsi1995,HM2002}. In particular, if $\sQ$ is the interval $[0,1]$ endowed with the G\"{o}del t-norm, i.e., $\sQ=([0,1],\wedge,1)$,   then a $\sQ$-cotopological space is strong if and only if it is a fuzzy neighborhood space in the sense of Lowen \cite{Lowen82}.

\begin{prop}\label{[0,1]-app = strong} Let $\sQ=([0,1],\with,1)$ with $\with$ being a  continuous t-norm.
Then for each $\sQ$-approach space $(X,\delta)$, $\kappa(X,\delta)$ is a strong $\sQ$-cotopological space. Furthermore, every strong $\sQ$-cotopological space arises in this way; explicitly,  for each strong $\sQ$-cotopological space $(X,\tau)$, $\tau=\kappa\circ\zeta(\tau)$.
Therefore,  the functor $\zeta$ restricts to an isomorphism between the subcategory consisting of strong $\sQ$-cotopological spaces and the category of $\sQ$-approach spaces, with $\kappa$ being its inverse.
\end{prop}
In the case that $\&$ is the product t-norm, Proposition \ref{[0,1]-app = strong} is in effect the characterization of approach spaces by lower regular functions (see e.g. \cite{RL97,Lowen15}).  In order to prove Proposition \ref{[0,1]-app = strong}, we prove a lemma first.
\begin{lem}\label{closure in a strong top} Let $\sQ=([0,1],\with,1)$ with $\with$ being a  continuous t-norm. Then for each strong $\sQ$-cotopological space $(X,\tau)$ and each $\lambda\in[0,1]^X$, the closure of $\lambda$ is given by \[\overline{\lambda}= \bv_{p\in[0,1]}p\with \overline{\lambda_{[p]}},\] where, \[ \lambda_{[p]}\colon X\lra[0,1], \quad \lambda_{[p]}(x)\coloneqq \begin{cases} 1, &\lambda(x)\geq p, \\ 0, & {\rm otherwise}. \end{cases}\]\end{lem}

\begin{proof}The conclusion is contained in \cite[Theorem 1.3]{HM2002},  it is essentially   the implication (iii) $\Rightarrow$ (v) therein. A direct proof is included here for the sake of completeness. For each integer $n\geq 2$, since \[\bv_{0\leq i\leq n-1}{ {\frac{i}{n}}\with\lambda_{[i/n]}}\leq \lambda \leq \bv_{0\leq i\leq n-1} {\frac{i+1}{n}}\with\lambda_{[i/n]}, \] it follows from (cl6) in Proposition \ref{closure:cl6}  that  \[\bv_{0\leq i\leq n-1}{ {\frac{i}{n}}\with\overline{\lambda_{[i/n]}}}\leq \overline{\lambda} \leq \bv_{0\leq i\leq n-1} {\frac{i+1}{n}}\with\overline{\lambda_{[i/n]}}. \] By uniform continuity of $\with\colon [0,1]\times[0,1]\lra[0,1]$, we have  \[\lim_{n\ra\infty}\bv_{0\leq i\leq n-1}{ {\frac{i}{n}}\with\overline{\lambda_{[i/n]}}}= \lim_{n\ra\infty} \bv_{0\leq i\leq n-1} {\frac{i+1}{n}}\with\overline{\lambda_{[i/n]}}. \] Therefore,  \[\overline{\lambda}= \bv_{p\in[0,1]}p\with \overline{\lambda_{[p]}},\] as desired.\end{proof}

The argument of Lemma \ref{closure in a strong top} implies that for each closed set $\lambda$ of a  strong $\sQ$-cotopological space $(X,\tau)$, it holds that \begin{equation} \label{subbasis} \lambda= \bw_{n\geq 2}\bv_{0\leq i\leq n-1} {\frac{i+1}{n}}\with\overline{\lambda_{[i/n]}}.\end{equation} So, \(\{\overline{1_A}\mid A\subseteq X\}\) is a subbasis of $\tau$. In other words, $\tau$ is the coarsest strong $\sQ$-cotopology on $X$ that has all $\overline{1_A}$ (closure of $1_A$ in $(X,\tau)$), $A\subseteq X$, as closed sets.

\begin{proof}[Proof of Proposition \ref{[0,1]-app = strong}] First, we show that  for each $\sQ$-approach space $(X,\delta)$, $\kappa(X,\delta)$ is a strong $\sQ$-cotopological space. To this end, we check that  if $\lambda\colon (X,\delta)\lra \mathbb{K}$ is continuous, then so is $p\with\lambda$ for all $p\in[0,1]$. This is easy since for each $x$ and each nonempty subset $A$ of $X$, \begin{align*}\delta(x,A)&\leq \delta_\mathbb{K}(\lambda(x),\lambda(A))\\ &= \inf\lambda(A) \ra\lambda(x) \\
&\leq (p\with\inf\lambda(A)) \ra(p\with\lambda(x))\\ &= \inf(p\with\lambda)(A)  \ra(p\with\lambda)(x) & \text{($\with$ is continuous)}\\ & = \delta_\mathbb{K}((p\with\lambda)(x),(p\with\lambda)(A)).\end{align*}

Next, we show that if $(X,\tau)$ is a strong $\sQ$-cotopological space, then $\tau=\kappa\circ\zeta(\tau)$; that is,    $\lambda\colon X\lra[0,1]$  is  a closed set of $(X,\tau)$ if and only if \[\lambda\colon (X,\zeta(\tau))\lra \mathbb{K}\] is a continuous map.

Assume that $\lambda$ is a closed set of $(X,\tau)$. Let $x$ be a point and $A$ be a nonempty subset of $X$. Set \[p=\inf\lambda(A).\] Then $p_A\leq \lambda$ and, by (cl6) in Proposition \ref{closure:cl6}, \[p\with \overline{1_A}(x) = \overline{p_A}(x)\leq\lambda(x).\] Hence \[\zeta(\tau)(x,A)=\overline{1_A}(x)\leq p\ra\lambda(x)= \delta_\mathbb{K}(\lambda(x),\lambda(A)),\] which shows that $\lambda\colon (X,\zeta(\tau))\lra \mathbb{K}$ is continuous.

Conversely, assume that $\lambda\colon (X,\zeta(\tau))\lra \mathbb{K}$ is continuous. Then for each point $x$  and each nonempty subset $A$ of $X$, \[\overline{1_A}(x)=\zeta(\tau)(x,A)\leq   \delta_\mathbb{K}(\lambda(x),\lambda(A)) = \inf\lambda(A)\ra\lambda(x),\] hence \[ \inf\lambda(A) \with\overline{1_A}(x)\leq\lambda(x). \] For each $p\in[0,1]$, letting $A=\{a\in X\mid \lambda(a)\geq p\}$ gives that  \[p\with \overline{\lambda_{[p]}}\leq\lambda.\] Consequently,  \[\bv_{p\in[0,1]}p\with \overline{\lambda_{[p]}}\leq\lambda,\] which implies that $\lambda$ is a closed set by the above lemma. \end{proof}

\begin{cor}\label{injective}
Let $\with$ be  a  continuous t-norm and let $\sQ=([0,1],\with,1)$. Then a separated   $\sQ$-approach space is injective if and only if it is a retract of some power of the space $\mathbb{K}$.
 \end{cor}

\begin{proof} Since the correspondence $(X,\delta)\mapsto(X, \kappa(\delta))$ is an isomorphism between the category of $\sQ$-approach spaces and that of strong $\sQ$-cotopological spaces, it suffices to check that  in the category of strong $\sQ$-cotopological spaces,  a $T_0$ space is injective if and only if it is a retract  of some  power  of $([0,1],\kappa(\delta_\mathbb{K}))$.

Since for each nonempty subset $A$ of $[0,1]$,  \[\delta_\mathbb{K}(-,A)=\inf A\ra {\rm id},\] it follows from equation (\ref{subbasis}) that $\kappa(\delta_\mathbb{K})$ is the  strong $\sQ$-cotopology on $[0,1]$ generated, as a subbasis, by the identity map. Then, by the argument of Proposition \ref{[0,1]-app = strong} one sees that a map $\lambda\colon X\lra[0,1]$ is a closed set of a strong $\sQ$-cotopological space $(X,\tau)$ if and only if $\lambda\colon (X,\tau)\lra([0,1],\kappa(\delta_\mathbb{K}))$ is continuous. With help of this fact, one readily verifies that in the category of strong $\sQ$-cotopological spaces, a $T_0$ space is  injective   if and only if it is a retract of some power of $([0,1],\kappa(\delta_\mathbb{K}))$.
\end{proof}

\begin{rem} An anonymous referee pointed out to us that Corollary  \ref{injective} can also be deduced  from existing results. The idea is sketched as follows:
\begin{enumerate}[label=(\alph*)] \setlength{\itemsep}{0pt}
\item
Since $[0,1]$ is completely distributive,      $[0,1]$-approach structures   can be equivalently described via $[0,1]$-valued ultrafilter convergence (see  \cite[Section 3]{LT17a}). Explicitly, given a $[0,1]$-approach space $(X,\delta)$, the degree that an ultrafilter $\mathfrak{x}$ converges to a point $x$ is $\bw_{A\in \mathfrak{x}}\delta(x,A).$ In particular, for the space $\mathbb{K}$, the degree that an ultrafilter $\mathfrak{x}$ converges to a point $x$ is
\(\bw_{A\in \mathfrak{x}}(\inf A\ra x).\)   \item Let $\sf U$ be the ultrafilter functor on the category of sets. Define $\xi\colon{\sf U}[0,1]\lra[0,1]$ by $$\xi(\mathfrak{x})=\bv_{A\in \mathfrak{x}}\inf A.$$ Then, by \cite[Examples 1.2\thinspace(4)]{HS2011}, for each continuous t-norm  the ultrafilter monad together with the map $\xi$ constitute a \emph{strict topological theory} $\mathcal{T}$ in the sense of \cite{Hof2007,Hof2011}. The $\mathcal{T}$-categories of this topological theory are essentially $[0,1]$-approach spaces; in particular, the space $\mathbb{K}$ is precisely the $\mathcal{T}$-category $([0,1],\hom_\xi)$ in \cite[Section 5]{Hof2007} and \cite[page 288]{Hof2011}.
\item
Since the $\mathcal{T}$-category $([0,1],\hom_\xi)$ is   injective  (see  \cite[page 293]{Hof2011} and \cite[Lemma 4.18]{HT2010}) and   initially dense (see \cite[Corollary 3.4]{HS2011}) in the category of $\mathcal{T}$-categories,   it follows that a separated $[0,1]$-approach space is injective if and only if it is a retract of some power of the space $\mathbb{K}$.  \end{enumerate}
\end{rem}

\begin{SA}
From now on, we always assume that $\sQ$ is the interval $[0,1]$ equipped with a continuous t-norm  $\with$, though some of the results hold for a more general quantale. The reason is that, for such a quantale,  injective and separated $\sQ$-approach spaces are precisely retracts of powers of the space $\mathbb{K}$ by Corollary \ref{injective}. In the subsequent sections we speak of $[0,1]$-approach space, $[0,1]$-order and complete $[0,1]$-lattice, instead of $\sQ$-approach space, $\sQ$-order and complete $\sQ$-lattice, respectively.\end{SA}

\section{Continuous $[0,1]$-lattices}
First, we recall some basic notions in domain theory \cite{Gierz2003}.

Let $X$ be a partially ordered set. Write $\idl X$ for the set of all ideals (= directed lower sets) of $X$ endowed with the inclusion order. Since for each $x\in X$, the lower set $\downarrow\!x$ is an ideal, we obtain a map (also called the Yoneda embedding): \[\sy\colon X\lra\idl X, \quad x\mapsto\thinspace\downarrow\!x.\]
We say that  \begin{enumerate}[label=(\roman*)] \setlength{\itemsep}{0pt}
\item $X$ is a directed complete partially ordered set (a dcpo for short) if the Yoneda embedding  $\sy\colon X\lra\idl X$ has a left adjoint  which necessarily maps each ideal to its join and is  denoted by $\sup\colon\idl X\lra X$. \item $X$ is a domain if $X$ is a dcpo and the left adjoint $\sup\colon\idl X\lra X$ of  $\sy\colon X\lra\idl X$ also has a left adjoint, denoted by $\thda\colon X\lra\idl X$.      In other words, $X$ is a domain if there is a string of adjunctions \[\thda \dashv \sup\dashv\sy\colon X\lra\idl X.\]
\item $X$ is a continuous lattice if it is at the same time a complete lattice and a domain. \end{enumerate}

Given a dcpo $X$ and two elements $x$ and $y$, we say that $x$ is \emph{way below} $y$,  in symbols  $x\ll y$, if for every directed set $D\subseteq X$, $y\leq \bv D$   implies that $x\leq d$ for some $d\in D$. It is known that a dcpo $X$ is a domain if and only if for all $a\in X$, $\thda  a\coloneqq\{x\in X\mid  x\ll a\}$ is an ideal and has $a$ as a join.

A map $f\colon X\lra Y$ between partially ordered sets is \emph{Scott continuous}   if it preserves  directed joins.

Let $X$ be a $[0,1]$-ordered set. A net $\{x_i\}_i$ of  $X$ is   said to be \emph{forward Cauchy} \cite{BvBR1998,FSW,Goubault} if
$$\bigvee_{i}\bigwedge_{i\leq j\leq l} X(x_j,x_l)=1.$$
A weight $\phi$ of $X$ is said to be \emph{forward Cauchy} if $$\phi=\bigvee_i\bigwedge_{i\leq j} X(-,x_j) $$ for some forward Cauchy net $\{x_i\}_i$ of $X$.

Forward Cauchy weights are an extension  of directed lower sets to the realm of $[0,1]$-ordered sets. It should be mentioned that  in the quantale-valued context  there exist different extensions of ``directed lower sets'',   a comparative study of  some of them can be found in  \cite{LZZ2020}.

There is a nice characterization of forward Cauchy weights, which says that such weights are precisely the \emph{inhabited} and \emph{irreducible} ones.

\begin{prop}{\rm (\cite[Theorem 3.13]{LZZ2020})} \label{FC=irr}
Let $X$ be a   $[0,1]$-ordered set. Then a weight $\phi$ of $X$ is forward Cauchy if and only if it satisfies:   \begin{enumerate}[label=\rm(\roman*)] \setlength{\itemsep}{0pt}
	\item $\phi$ is inhabited in the sense that $\bigvee_{x\in X}\phi(x)=1$;  \item $\phi$ is irreducible in the sense that  for all weights $\phi_1,\phi_2$ of $X$,
$$\sub_X(\phi,\phi_1\lor\phi_2)=\sub_X(\phi,\phi_1)\lor\sub_X(\phi,\phi_2). $$
\end{enumerate}
\end{prop}

For each $[0,1]$-ordered set $X$, we write \[\mathcal{D}X\] for the $[0,1]$-ordered set consisting of forward Cauchy weights with $[0,1]$-order inherited from $\CP X$. Since for each $[0,1]$-order-preserving map $f\colon X\lra Y$ and each forward Cauchy weight $\phi$ of $X$, $f^\ra(\phi)$ is a forward Cauchy weight of $Y$, the correspondence \[f\colon X\lra Y \quad \mapsto \quad f^\ra\colon\mathcal{D}X\lra \mathcal{D}Y\] defines a subfunctor \[\mathcal{D}\colon [0,1]\text{-}{\sf Ord}\lra [0,1]\text{-}{\sf Ord}\]   of $\mathcal{P}\colon [0,1]\text{-}{\sf Ord}\lra [0,1]\text{-}{\sf Ord}$. Furthermore, it is known (see e.g. \cite{FSW,LZ2020}) that $\mathcal{D}$    gives rise to a submonad \[\mathbb{D}=(\mathcal{D},\sfm,\sy)\] of the monad \[\mathbb{P}=(\mathcal{P},\sfm,\sy).\] That means,  $\mathcal{D}$ satisfies the following requirements: \begin{enumerate}[label=(\roman*)] \setlength{\itemsep}{0pt} \item For each $[0,1]$-ordered set $X$ and each $x\in X$, $\sy_X(x)\in \mathcal{D}X$. So $\sy$ is also a natural transformation from the identity  functor to $\mathcal{D}$. \item $\mathcal{D}$ is closed under multiplication in the sense that for each $[0,1]$-ordered set $X$ and each $\Phi\in \mathcal{D}\mathcal{D} X$, \[\sfm_X\circ(\varepsilon*\varepsilon)_X(\Phi)\in \mathcal{D}X,\] where $\varepsilon$ is the inclusion transformation of $\mathcal{D}$ in $\CP$ and $\varepsilon*\varepsilon$ stands for the horizontal composite of $\varepsilon$ with itself.
So  $\sfm$ induces a natural transformation $\mathcal{D}^2\lra \mathcal{D}$, which is also denoted by $\sfm$. \end{enumerate}

We say that a $[0,1]$-ordered set $X$ is \emph{$\mathcal{D}$-cocomplete} if  every forward Cauchy weight of $X$ has a supremum. It is clear that $X$ is $\mathcal{D}$-cocomplete  if and only if the map $\sy\colon X\lra\mathcal{D}X$ has a left adjoint, denoted by $\sup\colon\mathcal{D}X\lra X$, which necessarily maps each forward Cauchy weight to its supremum.

Bearing in mind that forward Cauchy weights are  an analogy of directed lower sets in the $[0,1]$-enriched context, we say that a $[0,1]$-order-preserving map $f\colon X\lra Y$ is \emph{$[0,1]$-Scott continuous} if $f$ preserves $\mathcal{D}$-suprema in the sense that for each forward Cauchy weight $\phi$ of $X$, $$f(\sup\phi)=\sup f^\ra(\phi)$$ whenever $\sup\phi$ exists.

Write $$[0,1]\text{-}\sf Ord^\uparrow $$  for the category of $[0,1]$-ordered sets and $[0,1]$-Scott continuous maps.
The full subcategory \[[0,1]\text{-}\sf DCPO\] of $[0,1]\text{-}\sf Ord^\uparrow $ composed of separated and $\mathcal{D}$-cocomplete $[0,1]$-ordered sets  is precisely the category of Eilenberg-Moore algebras of the monad \((\mathcal{D},\sfm,\sy),\)   it is an analogue in the $[0,1]$-enriched context of the category {\sf DCPO} of directed complete partially ordered sets and Scott continuous maps.

\begin{rem}\label{Yoneda limits} In the literature, $[0,1]$-Scott continuous maps are also said to be Yoneda continuous. The reason is  as follows.
An element $x$ of a $[0,1]$-ordered set $X$ is called a \emph{Yoneda limit}  of a forward Cauchy net $\{x_i\}_i$ \cite{BvBR1998,FSW,Goubault} if for all $y\in X$,
$$X(x,y)=\bv_i\bw_{j\geq i}X(x_j,y).$$
It is known (see e.g. \cite[Lemma 46]{FSW}) that $x$ is a Yoneda limit of a forward Cauchy net $\{x_i\}_i$  if and only if $x$ is a supremum of the forward Cauchy weight $\bv_i\bw_{j\geq i}X(-,x_j)$. So, a $[0,1]$-order-preserving map is  Yoneda continuous (i.e., preserves Yoneda limits) if and only if it is $[0,1]$-Scott continuous.
\end{rem}

\begin{prop}\label{underlying order of Yoneda} The underlying ordered set of each $\mathcal{D}$-cocomplete $[0,1]$-ordered set is directed complete.\end{prop}

\begin{proof}The conclusion is essentially \cite[Proposition 4.5]{LiZ18a} when $\with$ is isomorphic to the product t-norm. Here we present a simple proof for the general case. Let $X$ be a $\mathcal{D}$-cocomplete $[0,1]$-ordered set; and  let $D$ be a directed subset of $X_0$, viewed  as a forward Cauchy net in $X$. By assumption, the forward Cauchy   weight $\bv_{x\in D}\bw_{y\geq x}X(-,y)$ has a supremum, say $a$. Since $\bw_{y\geq x}X(-,y)=X(-,x)$ for all $x\in D$,   $a$ is a supremum of the   weight $\bv_{x\in D}X(-,x)$. From this one readily deduces  that $a$ is a join of $D$ in $X_0$. \end{proof}

The above argument makes use of the fact that for each   directed set  $D$ of $X_0$, the weight $\bigvee_{d\in D}X(-,d)$ is forward Cauchy. The following conclusion says that for a complete $[0,1]$-lattice $X$, all forward Cauchy weights of $X$ are of this form.

\begin{prop}{\rm(\cite[Proposition 4.8]{LZ2020})} \label{ideal}
Let $X$ be a complete $[0,1]$-lattice and let $\phi$ be a forward Cauchy weight of $X$. Then
$$\Lambda(\phi)=\{x\in X\mid \phi(x)=1\}$$
is an ideal of $X_0$  and
$$\phi=\bigvee_{x\in\Lambda(\phi)}X(-,x).$$
Furthermore, $\sup\phi$ is the join of $\Lambda(\phi)$ in $X_0$.
\end{prop}

\begin{cor}{\rm(\cite[Corollary 4.9]{LZ2020})} \label{Y=S}
A $[0,1]$-order-preserving map $f\colon X\lra Y$ between complete $[0,1]$-lattices is $[0,1]$-Scott continuous if and only if $f\colon X_0\lra Y_0$ is Scott continuous.
\end{cor}

\begin{defn}  Let $X$ be a separated $[0,1]$-ordered set. We say that \begin{enumerate}[label=\rm(\roman*)] \setlength{\itemsep}{0pt}
	\item (\cite{AW2011,HW2011,HW2012})
$X$ is a $[0,1]$-domain if there is a string of adjunctions \[{\sf d} \dashv{\sup} \dashv\sy \colon X\lra\mathcal{D}X.\]
\item $X$ is a continuous $[0,1]$-lattice if it is at the same time a complete $[0,1]$-lattice and a $[0,1]$-domain. \end{enumerate}
\end{defn}

\begin{prop} \label{underlying [0,1]-CL} The underlying order of a continuous $[0,1]$-lattice is a continuous lattice; that is, if $X$ is a  continuous $[0,1]$-lattice, then $X_0$ is a continuous  lattice.\end{prop}
\begin{proof} Contained in \cite[Proposition 4.11]{LZ2020}. \end{proof}

\begin{prop}  \label{chacl}
Let $X$ be a  complete $[0,1]$-lattice such that $X_0$ is a continuous  lattice. Then the following conditions are equivalent: \begin{enumerate}[label=\rm(\arabic*)] \setlength{\itemsep}{0pt}
	\item $X$ is a continuous $[0,1]$-lattice.
 \item For each $x\in X$ and each forward Cauchy weight $\phi$ of $X$,
	$$ X(x,\sup\phi)=\bigwedge_{y\ll x}\phi(y),$$
	where $\ll$ denotes the way below relation in $X_0$.
\item The map \[{\sf d}\colon X\lra \mathcal{D}X,\quad {\sf d}(x)=\bv_{y\ll x}X(-,y)\] preserves $[0,1]$-order, 	where $\ll$ denotes the way below relation in $X_0$.\end{enumerate}	
\end{prop}
\begin{proof} This   is a slightly different formulation of \cite[Proposition 4.11]{LZ2020}. We check $(3)\Rightarrow(1)$ here. By Proposition \ref{ideal} one sees that for each $x\in X$, $\bv_{y\ll x}X(-,y)$ is the smallest  forward Cauchy weight  of $X$ (under pointwise order) that has $x$ as a supremum, so, ${\sf d}\colon X_0\lra (\mathcal{D}X)_0$ is left adjoint to $\sup\colon  (\mathcal{D}X)_0\lra X_0$. Hence, ${\sf d}\colon X \lra \mathcal{D}X$ is left adjoint to $\sup\colon  \mathcal{D}X\lra X$ by Theorem \ref{Characterization of adjoints}. \end{proof}

\begin{exmp}  \label{dr is continuous} For each continuous t-norm $\with$, the $[0,1]$-ordered set $([0,1],\alpha_R)$ is a continuous $[0,1]$-lattice. Write $X$ for the $[0,1]$-ordered set $([0,1],\alpha_R)$.
By Proposition \ref{chacl}, it suffices to check that
the map
\[{\sf d}\colon X\lra \mathcal{D}X,\quad {\sf d}(t)(x) =\begin{cases}t\ra x, &t=1,\\ \bigvee\limits_{b>t}(b\ra x), &t<1 \end{cases} \]
preserves $[0,1]$-order. To this end, we only need to check that \begin{equation}\label{inequa1} s\ra t \leq \bw_{x\in[0,1]}\Big(\bigvee_{a>t}(a\ra x)\ra\bigvee_{b>s}(b\ra x)\Big) \end{equation} whenever $t<s<1$. For each $a>t$, let $b=(s\ra t)\ra a$. By continuity of $\with$ one readily verifies that $b>s$. Then, the inequality (\ref{inequa1})   follows from the fact that for each $x\in[0,1]$, \[(s\ra t)\with (a\ra x)\leq ((s\ra t)\ra a)\ra x =b\ra x.\]
\end{exmp}

\begin{exmp}\label{continuity of dl} For a continuous t-norm $\with$, the $[0,1]$-ordered set $([0,1],\alpha_L)$ is a continuous $[0,1]$-lattice if and only if $\with $ satisfies the condition (S). The conclusion is contained in \cite[Theorem 6.4]{LZ2020}, a direct verification is included here for convenience of the reader.

Write $X$ for the $[0,1]$-ordered set $([0,1],\alpha_L)$.
By Proposition \ref{chacl}, it suffices to show that the map \[{\sf d}\colon X\lra \mathcal{D}X,\quad   {\sf d}(t)(x) =\begin{cases}x\ra 0, &t=0,\\ \bigvee\limits_{b<t}(x\ra b), &t>0 \end{cases} \]
preserves $[0,1]$-order if and only if
$\with $ satisfies the condition (S). We check the only-if-part here and leave the verification of the if-part  to the reader.

Suppose that  $\with $ does not satisfy the condition (S). Then there exist idempotent elements $p,q>0$  such that $([p,q],\with ,q)$ is isomorphic to $([0,1],\with_{\L},1)$. Pick $t\in(p,q)$. Then \[{\sf d}(t)(x)= \bigvee_{b<t}(x\ra b)= x\ra t,\] hence  \[\mathcal{D}X({\sf d}(t), {\sf d}(p))= \bw_{x\in[0,1]}\Big((x\ra t)\ra\bv_{b<p}(x\ra b)\Big)\leq \bv_{b<p}(t\ra b)=p <t\ra p,\] which shows that ${\sf d}\colon X\lra \mathcal{D}X$ does not preserve  $[0,1]$-order.
\end{exmp}

For each $[0,1]$-ordered set $X$, the $[0,1]$-ordered set $\mathcal{D}X$ of forward Cauchy weights is a $[0,1]$-domain. This fact is an instance of   \cite[Proposition 3.3]{LZ2020} since $(\mathcal{D},\sfm,\sy)$ is a submonad of $(\mathcal{P},\sfm,\sy)$. But, in contrast to the fact that  the set of ideals of a complete lattice is always a continuous lattice, the $[0,1]$-ordered set $\mathcal{D}X$ of forward Cauchy weights of a complete $[0,1]$-lattice $X$ may fail to be a  complete $[0,1]$-lattice. Actually, in order that $\mathcal{D}X$  is a complete $[0,1]$-lattice for every complete $[0,1]$-lattice $X$,  it is necessary  and sufficient to require  that the t-norm $\with$ satisfies the condition (S), see Lemma 5.3 and Theorem 6.4 in \cite{LZ2020}.

While the underlying order of a  $\mathcal{D}$-cocomplete $[0,1]$-ordered set is  directed complete (Proposition \ref{underlying order of Yoneda}) and the underlying order of a continuous $[0,1]$-lattice is a continuous lattice (Proposition \ref{underlying [0,1]-CL}), the underlying order of a $[0,1]$-domain may fail to be a domain.

\begin{exmp}\label{underling order of q-domain}
In this example, we assume that the continuous t-norm $\with$ satisfies the condition (S). Let $X=[0,1]\cup\{\infty\}$ and let $$ X(x,y)= \begin{cases}x\ra y, & x,y\in [0,1], \\ 1, & x=y=\infty,\\  0, & x\in [0,1], y=\infty, \\  y, & x=\infty,y\in[0,1]. \end{cases} $$

Since forward Cauchy nets in $X$ are, essentially, the forward Cauchy nets in $([0,1],\alpha_L)$ plus constant nets with value $\infty$, it is clear that  $X$ is a separated and $\mathcal{D}$-cocomplete $[0,1]$-ordered set. Since $\with$ satisfies the condition (S), a routine calculation shows that the map \[{\sf d}\colon X\lra \mathcal{D}X, \quad {\sf d}(x)=\begin{cases}X(-,\infty), & x=\infty,\\    X(-,0), & x=0,\\ \bv\limits_{y<x}X(-,y),&x\in(0,1]\end{cases}\] is left adjoint to $\sup\colon\mathcal{D}X\lra X$, so $X$ is a $[0,1]$-domain. But,  $X_0$ is not a domain.
\end{exmp}

Let $X$ be a $\mathcal{D}$-cocomplete $[0,1]$-ordered set. The \emph{way below $[0,1]$-relation} (with respect to the class of weights $\mathcal{D}$) \cite{HW2011,HW2012,Was2009} of $X$ refers to the $[0,1]$-relation $w \colon X\times X\lra [0,1]$ defined by
$$w (x,y)=\bigwedge_{\phi\in\mathcal{D}X}X(y,\sup\phi)\ra\phi(x).$$
Some basic properties of the way below $[0,1]$-relation $w$: for all $x,y,z,u\in X$,
\begin{enumerate}[label=\rm(\roman*)] \setlength{\itemsep}{0pt}
	\item  $w(x,y)\leq X(x,y)$;
	\item  $w(y,z)\with X(x,y)\leq w(x,z)$, in particular, $w(-,z)\colon X
\lra ([0,1],\alpha_R)$ preserves $[0,1]$-order;
	\item  $X(z,u)\with w(y,z)\leq w(y,u)$, in particular, $w(y,-)\colon X
\lra ([0,1],\alpha_L)$ preserves $[0,1]$-order.
\end{enumerate}

\begin{prop} {\rm(\cite{HW2011,HW2012,Was2009})} \label{way below as left adjoint}
Let $X$ be a separated and $\mathcal{D}$-cocomplete $[0,1]$-ordered set. Then $X$ is a $[0,1]$-domain if and only if for all $x\in X$, the weight $w(-,x)$ is forward Cauchy with $x$ being a supremum. In this case, the left adjoint ${\sf d}\colon X\lra\mathcal{D}X$ of $\sup\colon\mathcal{D}X\lra X$ is given by ${\sf d}(x)=w(-,x)$.
\end{prop}

\begin{cor}\label{way-product}
Let $X$ be a continuous $[0,1]$-lattice. Then for all $x,y\in X$,
$$w(x,y)=\bigvee_{a\ll y}X(x,a)=\bigvee_{a\ll y}w(x,a),$$ where $\ll$ refers to the way below relation in the complete lattice $X_0$.
\end{cor}

\begin{proof} The first equality follows  immediately from Proposition \ref{chacl} and Proposition \ref{way below as left adjoint}. To see the second equality,  assume that $a\ll y$ in $X_0$. Pick some $c\in X$  such that $a\ll c\ll y$ in $X_0$. Then $X(x,a)\leq w(x,c)$ by the first equality, hence
$$ \bigvee_{a\ll y}X(x,a)\leq \bigvee_{a\ll y}w(x,a).$$ The converse inequality is trivial.
\end{proof}

The following conclusion, known as the interpolation property, has appeared in different contexts in the literature, see e.g. \cite{HW2011,HW2012,Stubbe2007}.

\begin{prop}\label{interpolation} The way below $[0,1]$-relation in a $[0,1]$-domain $X$ has the interpolation property in the sense that \[w(x,y)=\bv_{z\in X}w(z,y)\with w(x,z).\]\end{prop}


The following definition is a direct extension of that for generalized metric spaces (i.e., ordered sets valued in Lawvere's quantale)  in \cite{BvBR1998,Goubault}.
\begin{defn}
An element $a$ of a $[0,1]$-ordered set $X$  is    compact if
$$X(a,-)\colon X\lra([0,1],\alpha_L)$$
is $[0,1]$-Scott continuous.  \end{defn}

\begin{lem}\label{characterizing compact element} Let $X$ be a $\mathcal{D}$-cocomplete $[0,1]$-ordered set and let $a$ be an element of $X$. Then the following conditions are equivalent: \begin{enumerate}[label=\rm(\arabic*)]\setlength{\itemsep}{0pt}
\item    $a$   is  compact. \item For each forward Cauchy weight $\phi$ of $X$, $X(a,\sup\phi)=\phi(a)$.  \item $w(-,a)=X(-,a)$. \end{enumerate} \end{lem}

A   $\mathcal{D}$-cocomplete $[0,1]$-ordered set $X$  is  said to be  \emph{algebraic}  if its set $K(X)$ of compact elements is \emph{$\mathcal{D}$-codense} in the sense that for each $x\in X$, there is a forward Cauchy weight $\phi$ of $K(X)$ such that $x$ is a supremum of $i^\ra(\phi)$, where $i$ denotes the inclusion map $K(X)\hookrightarrow X$. In other words, a  $[0,1]$-ordered set $X$  is  {algebraic} if it is $\mathcal{D}$-cocomplete and each of its elements is a Yoneda limit (see Remark \ref{Yoneda limits} for definition) of a forward Cauchy net consisting of compact elements.

\begin{prop}\label{algebraic is continuous}  Every separated  algebraic $[0,1]$-ordered set is a $[0,1]$-domain. Such a $[0,1]$-domain is called an algebraic $[0,1]$-domain. \end{prop}

\begin{proof}Similar to the proof of \cite[Proposition 5.4]{LiZ18b}. \end{proof}

\begin{cor}\label{algebraicity of dr} For a continuous t-norm $\with$ on $[0,1]$, the following conditions are equivalent: \begin{enumerate}[label=\rm(\arabic*)]\setlength{\itemsep}{0pt}
\item $\with$ is Archimedean; that is $([0,1],\with,1)$ has no idempotent element other than $0$ and $1$. \item  The complete $[0,1]$-lattice $([0,1],\alpha_L)$ is algebraic. \item The complete $[0,1]$-lattice $([0,1],\alpha_R)$ is algebraic.\end{enumerate} \end{cor}

\begin{proof} This follows from a combination of Example \ref{dr is continuous}, Example  \ref{continuity of dl}, Lemma \ref{characterizing compact element} and Proposition \ref{algebraic is continuous}, details are left to the reader.
\end{proof}

We write \[[0,1]\text{-}{\sf ConLat}\] for the category with \begin{itemize}\setlength{\itemsep}{0pt} \item  objets:  continuous $[0,1]$-lattices; \item morphisms:    $[0,1]$-Scott continuous right adjoints. \end{itemize}

If the continuous t-norm $\with $ satisfies the condition (S),  then $\conlat$ is monadic over the category of sets (see  Theorem \ref{CL is monadic over Set} below).  In particular, $\conlat$ is a complete category. For an arbitrary continuous t-norm $\with$, we do not know  whether $[0,1]\text{-}{\sf ConLat}$ is monadic over sets, however, the following proposition says that it is  a complete category.

\begin{prop} \label{product}
The category $\conlat$ is complete.
\end{prop}
We make some preparations before proving this conclusion. First of all, we note that  from Proposition \ref{retractcomplete} and \cite[Proposition 3.4]{LZ2020} it follows that every retract of a $[0,1]$-domain (continuous $[0,1]$-lattice, resp.) in the category $[0,1]\text{-}\sf Ord^\uparrow$ is a $[0,1]$-domain (continuous $[0,1]$-lattice, resp.).

\begin{lem}\label{cc}
Suppose that $X,Y$ are separated $[0,1]$-ordered sets, and that $g \colon  X\lra Y$  is a $[0,1]$-Scott continuous right adjoint.
\begin{enumerate}[label=\rm(\roman*)] \setlength{\itemsep}{0pt}
\item If $g$ is surjective and  $X$ is a
  continuous $[0,1]$-lattice, then so is  $Y$.
\item If $g$ is injective  and  $Y$ is
 a continuous $[0,1]$-lattice, then  so is $X$.
\end{enumerate}
\end{lem}
\begin{proof}
(i) Let $f\colon Y\lra X$ be the left adjoint of $g$. Then $g\circ f=1_Y$. Thus,  as a retract of $X$ in $[0,1]\text{-}\sf Ord^\uparrow$, $Y$ is a continuous $[0,1]$-lattice.

(ii) Let $f\colon Y\lra X$ be the left adjoint of $g$. Then $f\circ g=1_X$. Thus, as a retract of $Y$ in $[0,1]\text{-}\sf Ord^\uparrow$, $X$ is a continuous $[0,1]$-lattice. \end{proof}

We say that a $[0,1]$-order-preserving map $k\colon X\lra X$ is a \emph{kernel operator} if $k^2=k$ and $k(x)\leq x$ in $X_0$ for all $x\in X$. If $k\colon X\lra X$ is a kernel operator, then $k\colon X\lra k(X)$ is right adjoint to the inclusion $k(X)\lra X$, so $k(X)$ is a retract of $X$ in $[0,1]\text{-}{\sf Ord}$. In particular, if $X$ is a continuous $[0,1]$-lattice and $k\colon X\lra X$ is a $[0,1]$-Scott continuous kernel operator, then  $k\colon X\lra k(X)$ is a $[0,1]$-Scott continuous right adjoint, hence $k(X)$ is a continuous $[0,1]$-lattice.

\begin{proof}[Proof of Proposition \ref{product}] We check that $\conlat$ has equalizers and products.  Given a parallel pair of morphisms $f,g\colon  X\lra Y$  in $\conlat$, let \[Z=\{x\in X\mid f(x)=g(x)\}.\] The inclusion $i\colon  Z \lra X $ is clearly   $[0,1]$-Scott continuous. Since $i\colon  Z_0\lra X_0$ is an equalizer of the pair $f,g\colon X_0\lra Y_0$ in the category of continuous lattices, it is a right adjoint between partially ordered sets. Thus, by Theorem \ref{Characterization of adjoints} we obtain that  $i\colon  Z \lra X $ is a $[0,1]$-Scott continuous right adjoint. Therefore,  $Z$ is a continuous $[0,1]$-lattice by the above lemma, so $i\colon  Z \lra X $ is an equalizer of $f,g\colon X\lra Y$ in $\conlat$.

To see that $\conlat$ has products, we check that
for each family $\{X_i\}_{i\in I}$ of continuous $[0,1]$-lattices, the product   $\prod_iX_i$ is a continuous $[0,1]$-lattice.

First of all, we note that by Proposition \ref{ideal},     the equality in condition (2) in Proposition \ref{chacl} is equivalent to the requirement that for each $x$ of $X$ and each ideal $D$ of $X_0$,
\begin{equation} \label{continuity via ideal}
	 X(x,{\textstyle \bigvee} D)=\bigwedge_{y\ll x}\bigvee_{d\in D}X(y,d).
\end{equation}
So,  we only need to  check that for each $\vec{x}=(x_i)_i\in\prod_i X_i$ and each ideal $D$ in the underlying order of $\prod_i X_i$,
\begin{equation} \label{L=R}\bw_i X_i(x_i,a_i)=\bigwedge_{\vec{y}\ll \vec{x}}\bigvee_{\vec{d}\in D}\bw_i X_i(y_i,d_i),\end{equation}
where $\vec{a}=(a_i)_i$ is the join of $D$ in the underlying order of $\prod_iX_i$.

Assume that $r<\bw_i X_i(x_i,a_i)$ and $\vec{y}\ll \vec{x}$. We are to find an element $\vec{c}$ of $D$ such \(r\leq\bw_i X_i(y_i,c_i)\), which proves the inequality $\leq$ in   (\ref{L=R}).
Write $\bot$ for the bottom element in the underlying order of each $X_i$. Since $\vec{y}\ll \vec{x}$, it follows that $y_i\ll x_i$ for all $i\in I$ and $y_i=\bot$ for all but a  finite number of $i$. If $y_i=\bot$, put $c_i=\bot$. If $y_i\not=\bot$,  because of $X_i$ being a continuous $[0,1]$-lattice,  we have \[X_i(x_i,a_i)\leq\bigvee_{b_i\ll a_i}X_i(y_i,b_i)\]  by  (\ref{continuity via ideal}), hence there is some $b_i\ll a_i$  such that $r\leq X_i(y_i,b_i)$; put $c_i$ to be this $b_i$. Then   $\vec{c}=(c_i)_i$   satisfies the requirement.

For the converse inequality, first  we note   that for each $i$, $a_i$ is the join of the directed set $\{d_i\}_{\vec{d}\in D}$ in the underlying order of $X_i$. Since $X_i$ is a continuous $[0,1]$-lattice, it follows from (\ref{continuity via ideal}) that
$$X_i(x_i,a_i)=\bigwedge_{z\ll x_i}\bigvee_{\vec{d}\in D}X_i(z,d_i).$$
For each $z\ll x_i$, define an element $\vec{y}=(y_j)_j$ by letting $y_i=z$ and $y_j=\bot$ for $j\neq i$. Then $\vec{y}\ll \vec{x}$ and for all $\vec{d}\in D$,
$$\bw_jX_j(y_j,d_j)=X_i(z,d_i),$$
and consequently, \begin{align*}\bw_i X_i(x_i,a_i)&= \bw_i \bigwedge_{z\ll x_i}\bigvee_{\vec{d}\in D}X_i(z,d_i) \geq\bigwedge_{\vec{y}\ll \vec{x}}\bigvee_{\vec{d}\in D}\bw_i X_i(y_i,d_i).\qedhere \end{align*}
\end{proof}

\begin{prop}\label{SCL} For a continuous t-norm $\with$ satisfying the condition (S), a $[0,1]$-ordered set $X$ is a  continuous $[0,1]$-lattice if and only if it is a retract of some power of $([0,1],\alpha_L)$ in the category $[0,1]\text{-}\sf Ord^\uparrow$. \end{prop}
\begin{proof}  Sufficiency  follows from Example \ref{continuity of dl}, Proposition \ref{product} and the fact that every retract  of a continuous $[0,1]$-lattice in  $[0,1]\text{-}\sf Ord^\uparrow$ is a continuous $[0,1]$-lattice. As for necessity, assume that $X$ is a  continuous $[0,1]$-lattice. By  \cite[Corollary 3.5]{LZ2020} $X$ is a retract of $\mathcal{D}X$ in category $[0,1]\text{-}\sf DCPO$.  Since the inclusion $\mathcal{D}X\lra\mathcal{P}X$ is $[0,1]$-Scott continuous,  $\mathcal{D}X$ is a retract of $\mathcal{P}X$ in $[0,1]\text{-}\sf DCPO$ by  \cite[Theorem 6.4\thinspace(6)]{LZ2020}. Finally, since $\mathcal{P}X$ is a retract of $([0,1],\alpha_L)^X$ in the category $[0,1]\text{-}\sf Sup$, it follows that $X$ is a retract  of  $([0,1],\alpha_L)^X$ in   $[0,1]\text{-}\sf   Ord^\uparrow$. \end{proof}

It is well-known that the category of continuous lattices is monadic over sets, see e.g. \cite{Day,Wyler}. This is still true for continuous $[0,1]$-lattices provided that the continuous t-norm $\&$ satisfies the condition (S).
\begin{thm}\label{CL is monadic over Set} If $\with $ satisfies the condition (S), then the category $\conlat$ is monadic over  the category of sets. \end{thm}

\begin{proof}Since $\with $ satisfies the condition (S),  it follows from Corollary 5.6 and Theorem 6.4 in \cite{LZ2020} that the forgetful functor \(\conlat\lra [0,1]\text{-}{\sf Ord}\) is monadic, hence a right adjoint. It is clear that the forgetful functor $[0,1]\text{-}{\sf Ord}\lra {\sf Set}$ is also a right adjoint, so the forgetful functor $\conlat\lra {\sf Set}$, as a composite of right adjoints, is a right adjoint. It remains, by Beck's theorem (see e.g. \cite[page 151]{MacLane1998}),  to check that the forgetful functor \(\conlat\lra {\sf Set}\) creates split coequalizers. This is contained in Proposition \ref{creates split coequalizer} below. \end{proof}

Let $X$ be a $[0,1]$-ordered set. Given $y\in X$   and let $p\in[0,1]$, a \emph{cotensor of $p$ and $y$}
(see e.g. \cite[page 288]{Stubbe2006}) is  an element   $p\rightarrowtail y$ of $X$   such that   for all $x\in X$, \[X(x,p\rightarrowtail y)=p\ra X(x,y).\] Some useful facts about cotensors are listed below: \begin{itemize} \setlength{\itemsep}{0pt} \item  If $X$ is a  complete $[0,1]$-lattice, then   all cotensors $p\rightarrowtail y$ exist. \item  Every right adjoint $f$ preserves cotensor; that is, $f(p\rightarrowtail y)$ is a cotensor of $p$ and $f(y)$. \item A map $f\colon X\lra Y$ between complete $[0,1]$-lattices preserves the $[0,1]$-order  if and only if $f\colon X_0\lra Y_0$ preserves order and $f(p\rightarrowtail x)\leq  p\rightarrowtail f(x)$ for all $x\in X$ and $p\in [0,1]$.  \end{itemize}

The verification of the following lemma is routine.
\begin{lem}\label{cong on CL}
Suppose that $X$ be  a  complete $[0,1]$-lattice;  suppose that   $R$ is a relation on    $X$ subject to the following conditions:
\begin{enumerate}[label=\rm(\roman*)]\setlength{\itemsep}{0pt} \item     $R$ is closed w.r.t. directed joins in $X_0 \times  X_0$;
\item   $R$ is closed w.r.t. meets in $X_0\times  X_0$;
\item If $(x,y)\in R$, then  $(p\rightarrowtail x, p\rightarrowtail  y)\in R$ for all $p\in[0,1]$.
\end{enumerate} Then, the map \[k\colon   X\lra X, \quad k(x)=\bw\{y\mid (x,y)\in R\}\]   is a $[0,1]$-Scott continuous kernel operator.
\end{lem}

\begin{prop}\label{creates split coequalizer} The forgetful functor $\conlat\lra {\sf Set}$  creates split coequalizers. \end{prop}

\begin{proof}
Let $f,g\colon  X\to Y$ be a parallel pair of morphisms in $\conlat$; and let  $h\colon  Y\lra Z$ be a  split coequalizer of $f,g$ in  ${\sf Set}$. By definition there exist morphisms  $Z\to^i Y\to^j X $ in ${\sf Set}$ such that
$$h\circ f=h\circ g,~ f\circ j=\id,~ h\circ i=\id,~ g\circ j=i\circ h.$$
Let \[R=\{(y_1,y_2)\in Y\times Y\mid  h(y_1)=h(y_2)\}.\] It is not hard to check that $(y_1,y_2)\in R$ if and only if there is some $(x_1,x_2)\in X\times X$ such that $g(x_1)=g(x_2)$, $y_1=f(x_1)$ and $y_2=f(x_2)$. With help of this fact, one readily verifies that $R$ satisfies the conditions (i)-(iii) in Lemma \ref{cong on CL},
hence  $R$ determines a $[0,1]$-Scott continuous kernel operator $k\colon  Y\lra  Y$. Since $k(Y)$ is a  continuous $[0,1]$-lattice with an underlying   set   equipotent to $Z$,   $Z$ can be made into a continuous $[0,1]$-lattice  so that $h\colon  Y\lra Z$ is a $[0,1]$-Scott continuous right adjoint. This proves that  the forgetful functor $\conlat\lra {\sf Set}$ creates split coequalizers. \end{proof}

\section{Injective $[0,1]$-approach spaces}

In this section we introduce the notion of Scott $[0,1]$-approach structure for $[0,1]$-ordered sets and prove the main result of this paper --- Theorem \ref{preserve product}. This theorem says that the structure of a separated injective $[0,1]$-approach space is completely determined by its specialization $[0,1]$-order which turns out to be a continuous $[0,1]$-lattice.

\begin{defn} (\cite[Definition 4.4]{Wagner97}) A weight $\phi$ of a $[0,1]$-ordered set $X$ is  Scott closed  if, as a  $[0,1]$-order-preserving map,  $\phi\colon  X\lra ([0,1],\alpha_R)$   is $[0,1]$-Scott continuous.  \end{defn}

\begin{lem}\label{Scott closed weights} Let $\phi$ be a weight of    a $[0,1]$-ordered set $X$. Then  $\phi$ is   Scott closed if and only if $$\sub_X(\lambda,\phi)= \phi(\sup\lambda)$$ for every forward Cauchy weight $\lambda$ of $X$.     \end{lem}
\begin{proof}
Routine, for example,
the necessity follows from Equation (\ref{inclusion as sup}). \end{proof}

\begin{prop}{\rm(\cite[Proposition 5.9]{LZZ2020})}
For each $[0,1]$-ordered set $(X,\alpha)$, the set $\sigma(\alpha)$ of Scott closed weights of $(X,\alpha)$  is a strong $[0,1]$-cotopology on $X$.
\end{prop}

Because of Proposition \ref{[0,1]-app = strong},
for each $[0,1]$-ordered set $(X,\alpha)$  we view the strong $[0,1]$-cotopology   $ \sigma(\alpha) $ and the $[0,1]$-approach structure  $ \zeta(\sigma(\alpha)) $  as different facets of the same object  as in the theory of approach spaces \cite{RL97,Lowen15}, and call them the \emph{Scott $[0,1]$-cotopology} and the \emph{Scott $[0,1]$-approach structure} of $(X,\alpha)$, respectively.
For each $[0,1]$-ordered set $(X,\alpha)$, we write $$\Sigma(X,\alpha),$$ or simply $\Sigma X$ if  $\alpha$ is clear from the context, for the $[0,1]$-cotopological space $(X,\sigma(\alpha))$  or the $[0,1]$-approach space  $(X,\zeta(\sigma(\alpha)))$.

\begin{rem}Since Lawvere's quantale $([0,\infty]^{\rm op},+,0)$ is isomorphic to the quantale   $([0,1],\&_P,1)$, the  Scott approach distances of generalized metric spaces in \cite{LiZ18b,Windels} are the Scott $[0,1]$-approach structures with respect to the product t-norm.\end{rem}

\begin{exmp}\label{complement} (\cite[Lemma 4.13]{Wagner97}) For each element $a$ of a $[0,1]$-ordered set, the weight $X(-,a)$ is Scott closed. This is because that for each  forward Cauchy weight $\lambda$ of $X$, \[\sub_X(\lambda,X(-,a))= X(\sup\lambda,a) \] by definition of $\sup\lambda$. It is clear that $X(-,a)$ is the closure of $1_a$ in $\Sigma X$.

In particular, for each $p\in[0,1]$, ${\rm id}\ra p$ is a Scott closed weight of $([0,1],\alpha_L)$, hence \[{\rm id}\ra p\colon ([0,1],\alpha_L)\lra([0,1],\alpha_R)\] is $[0,1]$-Scott continuous. \end{exmp}

If $f\colon X\lra Y$ is $[0,1]$-Scott continuous, then $f\colon \Sigma X\lra\Sigma Y$ is continuous. So we obtain a functor:
$$\Sigma\colon [0,1]\text{-}{\sf Ord}^\uparrow\lra[0,1]\text{-}{\sf App}.$$

\begin{prop}\label{YC=C} The functor $\Sigma$ is full; that is, a map $f\colon X\lra Y$ between  $[0,1]$-ordered sets is $[0,1]$-Scott continuous if and only if $f\colon \Sigma X\lra\Sigma Y$ is continuous.
\end{prop}

\begin{proof} This is  in essence    \cite[Proposition 4.15]{Wagner97}. \end{proof}

For each $[0,1]$-ordered set $X$, the specialization $[0,1]$-order of the $[0,1]$-approach space $\Sigma X$ coincides with that of $X$; that is, $\Omega\Sigma X=X$. In particular, the functor $$\Sigma\colon [0,1]\text{-}{\sf Ord}^\uparrow\lra[0,1]\text{-}{\sf App}$$ is a full embedding.

\begin{exmp}Since a map $\phi\colon  ([0,1],\alpha_R)\lra ([0,1],\alpha_R)$   is $[0,1]$-Scott continuous if and only if it is right continuous and preserves the $[0,1]$-order, it follows from Example \ref{cotopology of K} that   $$\Sigma([0,1],\alpha_R)= \mathbb{K}.$$ Thus, the Scott $[0,1]$-cotopology of $([0,1],\alpha_R)$ is the strong $[0,1]$-cotopology on $[0,1]$ generated by the identity map as a subbasis. This provides a simple proof of the assertion of \cite[Example 5.10]{LZZ2020}.
\end{exmp}

A closed set $\phi$ of a $[0,1]$-cotopological space $X$ is  \emph{irreducible} if    for all closed sets $\psi_1, \psi_2$ of $X$, \[\sub_X(\phi,\psi_1\vee\psi_2)= \sub_X(\phi,\psi_1)\vee \sub_X(\phi, \psi_2).\]

A $[0,1]$-cotopological space $X$ is   \emph{sober} \cite{Zhang2018} if for each inhabited (i.e., $\bv_{x\in X}\phi(x)=1$) and irreducible closed set $\phi$ of $X$, there is a unique $x\in X$ such that $\phi$ is the closure of $1_x$.

We say that a $[0,1]$-approach space $(X,\delta)$ is \emph{sober} if the $[0,1]$-cotopological space $(X,\kappa(\delta))$ is sober. If the continuous t-norm $\with$ is isomorphic to the product t-norm, then  this notion   is, in essence, that of sober approach spaces in \cite{BRC} postulated in terms of Lawvere's quantale.

From Proposition \ref{FC=irr} and \cite[Proposition 3.10]{Zhang2018}  one   deduces that the specialization $[0,1]$-order of a sober $[0,1]$-cotopological space is $\mathcal{D}$-cocomplete. Conversely,   the Scott $[0,1]$-cotopology of every $[0,1]$-domain is sober, as we now see.

\begin{lem}\label{Scott closure} Let $X$ be a $[0,1]$-domain and let $\phi$ be a  weight   of $X$. Then the closure of $\phi$ in $\Sigma X$ is given by \(\overline{\phi} (a)= \sub_X(w(-,a),\phi)\) for all $a\in X$, where $w$ is the way below $[0,1]$-relation of $X$.  \end{lem}

\begin{proof}First, we show that  $\phi$ is Scott closed if and only if  for all $a\in X$, \[ \phi (a)= \sub_X(w(-,a),\phi).\] Necessity follows   from Lemma \ref{Scott closed weights}  and  that $w(-,a)$ is a forward Cauchy weight with $a$ being a supremum by Proposition \ref{way below as left adjoint}. For sufficiency, we check that for each forward Cauchy weight $\lambda$, $\sub_X(\lambda,\phi)\leq \phi(\sup\lambda)$. Let $a=\sup\lambda$. Since $w(-,a)$ is a forward Cauchy weight and $w(-,a)\leq\lambda$, then $\sub_X(\lambda,\phi)\leq \sub_X(w(-,a),\phi) =\phi(\sup\lambda)$, as desired.

Next, we prove the conclusion. Write $\psi$ for the map $X\lra[0,1]$ defined by \[\psi(a)=  \sub_X(w(-,a),\phi).\] It is easily verified that (i) $\psi$ is a weight of $X$; (ii) $\phi\leq\psi$; and (iii)  $\psi\leq\lambda$ for any Scott closed weight $\lambda$ that majorizes $\phi$. To finish the proof, it remains to check that  $\psi$ is Scott closed. By the interpolation property of the  way below $[0,1]$-relation $w$, one   verifies that  $\psi$ satisfies that \( \psi (a)= \sub_X(w(-,a),\psi)\) for all $a\in X$, hence it is Scott closed.    \end{proof}

\begin{prop}\label{domain is sober} For each $[0,1]$-domain $X$, the   space $\Sigma X$ is sober.\end{prop}
\begin{proof}  Let $\lambda$ be an inhabited and irreducible closed set of $\Sigma X$. We have to show that there is  a unique element $b$ of $X$ such that $\lambda$ is the closure of $1_b$, i.e., $\lambda=X(-,b)$. Uniqueness is obvious since $X$ is separated. Now we prove the existence. Let \[\Downarrow\!\lambda=\bv_{a\in X}\lambda(a)\with w(-,a),\] where $w$ is the way below $[0,1]$-relation of $X$. We   show in two steps that $\Downarrow\!\lambda$ is a forward Cauchy weight of $X$ and its supremum satisfies the requirement.

\textbf{Step 1}. We use the characterization in Proposition \ref{FC=irr} to show that $\Downarrow\!\lambda$ is a forward Cauchy weight of $X$.

For each $a\in X$,  since $w(-,a)$ is a forward Cauchy weight, then  \(\bv_{x\in X}w(x,a)=1\) by Proposition \ref{FC=irr}.  Thus, \[\bv_{x\in X}\Downarrow\!\lambda(x)= \bv_{a\in X}\bv_{x\in X}\lambda(a)\with w(x,a)=\bv_{a\in X}\lambda(a)=1,\] showing that $\Downarrow\!\lambda$ is inhabited.

To see that $\Downarrow\!\lambda$ is irreducible, first we show that for each weight $\phi$ of $X$, \[\sub_X(\Downarrow\!\lambda, \phi)=\sub_X(\lambda, \overline{\phi}).\]  We calculate: \begin{align*}\sub_X(\Downarrow\!\lambda, \phi) &= \bw_{a\in X}(\lambda(a)\ra \sub_X(w(-,a),\phi) )\\ &= \bw_{a\in X}(\lambda(a)\ra \overline{\phi}(a) )\\ &= \sub_X(\lambda, \overline{\phi}).  \end{align*} Then, for any  weights $\phi_1,\phi_2$   of $X$, we have
\begin{align*}\sub_X(\Downarrow\!\lambda, \phi_1\vee\phi_2)  &= \sub_X(\lambda, \overline{\phi_1 \vee \phi_2}) \\  &= \sub_X(\lambda, \overline{\phi_1}\vee\overline{\phi_2}) \\ &= \sub_X(\lambda, \overline{\phi_1})\vee\sub_X(\lambda,\overline{\phi_2}) \\ &= \sub_X(\Downarrow\!\lambda,\phi_1)\vee\sub_X(\Downarrow\!\lambda,\phi_2), \end{align*}  hence $\Downarrow\!\lambda$ is irreducible.

\textbf{Step 2}. Since $\Downarrow\!\lambda$ is a forward Cauchy weight, it has a supremum, say $b$. We claim that $b$ satisfies the requirement.

On one hand, since $\lambda$ is   Scott closed,  then \(1=\sub_X(\Downarrow\!\lambda,\lambda)=\lambda(\sup\Downarrow\!\lambda)= \lambda(b),\) hence $X(-,b)\leq\lambda$.
On the other hand, since for each $a\in X$, \[\sub_X(w(-,a),\Downarrow\!\lambda)\geq \sub_X(w(-,a),\lambda(a)\with w(-,a))\geq \lambda(a),\] it follows from Lemma \ref{Scott closure} that $\lambda$ is the closure of $\Downarrow\!\lambda$, hence $\lambda\leq X(-,b)$ because $X(-,b)$ is Scott closed and contains $\Downarrow\!\lambda$. \end{proof}

Now we   present the main result in this paper.

\begin{thm}\label{preserve product}
If $X$ is a separated and injective $[0,1]$-approach space, then $\Omega(X)$ is a continuous $[0,1]$-lattice and $X=\Sigma\Omega(X)$.
\end{thm}

Before proving this theorem, we prove four lemmas first.

\begin{lem}\label{w(x,-)}
Let $X$ be a  $[0,1]$-domain. Then for each $x\in X$, $w(x,-)\colon X\lra ([0,1],\alpha_L)$ is $[0,1]$-Scott continuous.
\end{lem}

\begin{proof}This follows from the
facts that ${\sf d}\colon X\lra\mathcal{D}X$ is a left adjoint and hence preserves suprema, and the projection map $$ \mathcal{D}X\lra ([0,1],\alpha_L), \quad \phi\mapsto \phi(x)$$ preserves $\mathcal{D}$-suprema. This nice argument is pointed out  by an anonymous referee.\end{proof}


\begin{lem}\label{basis}
Let $X$ be a   $[0,1]$-domain.
Then $\{w(x,-)\rightarrow p\mid x\in X, p\in[0,1]\}$ is a subbasis for the closed sets of   $\Sigma X$.\end{lem}


\begin{proof}Since   ${\rm id}\ra p\colon ([0,1],\alpha_L)\lra([0,1],\alpha_R)$ is $[0,1]$-Scott continuous by Example \ref{complement} and $$w(x,-)\colon X\lra ([0,1],\alpha_L)$$ is $[0,1]$-Scott continuous by Lemma \ref{w(x,-)}, it follows that for each $p\in[0,1]$ the map $w(x,-)\ra p$, being a composite of $w(x,-)$ and ${\rm id}\ra p$,  is $[0,1]$-Scott continuous, hence a Scott closed weight of $X$.

Now we  show that
$$\phi(y)=\bigwedge_{x\in X}w(x,y)\rightarrow\phi(x)$$ for each Scott closed weight $\phi$ and each $y\in X$, which  entails the conclusion.

Since $\phi\colon X\lra([0,1],\alpha_R)$ is $[0,1]$-Scott continuous, it follows that \begin{align*}\phi(y)&= \phi(\sup w(-,y))\\ & = \sup\phi^\ra(w(-,y))\\ &=\sub_X(w(-,y),\phi) & \text{(Equation (\ref{inclusion as sup}))}\\ &=\bigwedge_{x\in X}w(x,y)\rightarrow\phi(x),\end{align*}
which completes he proof.
\end{proof}

\begin{lem}\label{way below in prod}
Let $X$ be a continuous $[0,1]$-lattice and $I$ be a set. Then for all $\vec{x}=(x_i)_i$ and $\vec{y}=(y_i)_i$ of $X^I$, $w(\vec{x},\vec{y})\leq\bigwedge\limits_{i\in I}w(x_i,y_i)$. Furthermore, if   $x_i$ is the bottom element of $X_0$ for all but a finite number of $i$, then $w(\vec{x},\vec{y})=\bigwedge\limits_{i\in I}w(x_i,y_i)$.
\end{lem}

\begin{proof}
By Proposition \ref{product}, $X^I$ is a continuous $[0,1]$-lattice, hence
$$w(\vec{x},\vec{y})=\bigvee_{\vec{a}\ll \vec{y}}X^I(\vec{x},\vec{a})$$
For each $\vec{a}\ll \vec{y}$ and $j\in I$, we have $a_j\ll y_j$, so
$$X^I(\vec{x},\vec{a})\leq X(x_j,a_j)\leq\bigvee_{a\ll y_j}X(x_j,a)=w(x_j,y_j),$$
which implies that
$$w(\vec{x},\vec{y})\leq\bigwedge\limits_{i\in I} w(x_i,y_i).$$

Now suppose there is a finite set $J\subseteq I$, such that, $x_i$ is the bottom element of $X_0$  for all $i\in I\setminus J$. Then
$$\bigwedge_{i\in I}w(x_i,y_i)=\bigwedge_{j\in J}w(x_j,y_j)=\bigwedge_{j\in J}\bigvee_{b_j\ll y_j}X(x_j,b_j).$$
For each $j\in J$ and $b_j\ll y_j$, let $\vec{c}=(c_j)_{j\in I}$, where $c_j$ equals $b_j$ if $j\in J$  and  is the bottom element of $X_0$ otherwise. Then $\vec{c}\ll\vec{y}$, thus,
\begin{align*}\bigwedge_{i\in I}w(x_i,y_i)&\leq\bigvee_{\vec{b}\ll\vec{y}}X^I(\vec{x},\vec{b}) =w(\vec{x},\vec{y}).\qedhere \end{align*}
\end{proof}

\begin{lem}\label{Sigma preserves product} For each  set $I$, $\Sigma (([0,1],\alpha_R)^I) =\mathbb{K}^I$.
\end{lem}

\begin{proof} We prove a slightly more general result: for each continuous $[0,1]$-lattice $X$ and   each  set $I$, $$\Sigma(X^I)=(\Sigma X)^I.$$
By Lemma \ref{basis}, it suffices to check that for each $\vec{x}\in X^I$ and $p\in[0,1]$, $w(\vec{x},-)\ra p$ is a closed set of the product space $(\Sigma X)^I$. We do this by showing that $w(\vec{x},-)\ra p$ is a meet of closed sets in the product space.

Given $\vec{y}\in X^I$ and $r< w(\vec{x},\vec{y})$,   by Corollary \ref{way-product} we have
$$w(\vec{x},\vec{y})=\bigvee\limits_{\vec{a}\ll\vec{y}}w(\vec{x},\vec{a}), $$
so there is some $\vec{a^r}\ll\vec{y}$ such that $r\leq w(\vec{x},\vec{a^r})$. Since $\vec{a^r}\ll\vec{y}$, there is  a finite subset $J$ of $I$  such that $a^r_i$ is the bottom element of $X_0$ for all $i\in I\setminus J$. Let
$$\lambda^r_{\vec{y}} =r\land\bigwedge\limits_{i\in J}w(x_i,-).$$
Then $\lambda^r_{\vec{y}}\ra p$, which is equal to $$(r\ra p)\vee \bv\limits_{i\in J}(w(x_i,-)\ra p),$$ is a closed set of the product space $(\Sigma X)^I$. Furthermore, it holds that \begin{enumerate}[label=\rm(\roman*)]\setlength{\itemsep}{0pt} \item $r= \lambda^r_{\vec{y}}(\vec{y})$; and \item $\lambda^r_{\vec{y}}\leq w(\vec{x},-)$. \end{enumerate} (i) is obvious. To see (ii), let $\vec{x'}$ be the point with $x'_i=x_i$ for   $i\in J$  and $x'_i=\bot$ for $i\notin J$. Then for all $\vec{z}\in X^I$,
\begin{align*}
\lambda^r_{\vec{y}}(\vec{z})&=r\land w(\vec{x'},\vec{z}) &\text{(Lemma \ref{way below in prod})} \\ &=r\land\bigvee\limits_{\vec{b}\ll\vec{z}}X^I(\vec{x'},\vec{b})\\ &=\bigvee\limits_{\vec{b}\ll\vec{z}}\Big(r\land\bigwedge\limits_{i\in J}X(x_i,b_i)\Big) \\ &\leq\bigvee_{\vec{b}\ll\vec{z}}\bigwedge_{i\in I}X(x_i,b_i)\\ &=w(\vec{x},\vec{z}).
\end{align*}

From (i) and (ii) it follows that \[w(\vec{x},-) =\bv\{\lambda^r_{\vec{y}}\mid r< w(\vec{x},\vec{y}),\vec{y}\in X^I\}.\]
Therefore, being equal to the meet of the family   \[\{\lambda^r_{\vec{y}} \ra p \mid r< w(\vec{x},\vec{y}),\vec{y}\in X^I\}\] of closed sets, $w(\vec{x},-)\ra p$ is  closed.
\end{proof}

\begin{rem}Lemma \ref{Sigma preserves product} was proved in \cite{LiZ18b} when $\with$ is isomorphic to the product t-norm.  The proof therein depends on the fact that for such a t-norm, $([0,1],\alpha_R)$ is an algebraic $[0,1]$-domain. But, Corollary \ref{algebraicity of dr} says that for a general  continuous t-norm, $([0,1],\alpha_R)$ is seldom algebraic.\end{rem}

\begin{proof}[Proof of Theorem \ref{preserve product}]
By Corollary \ref{injective}, there exist continuous maps $s\colon X\lra\mathbb{K}^I$ and $r\colon\mathbb{K}^I\lra X$ such that $r\circ s=1_X$. Since $\Omega(\mathbb{K}^I)=([0,1],\alpha_R)^I$, it follows from Proposition \ref{retractcomplete} that $\Omega(X)$ is a complete $[0,1]$-lattice. To see that $\Omega(X)$ is a $[0,1]$-domain, it suffices to check that both $r\colon ([0,1],\alpha_R)^I\lra\Omega(X)$ and $s\colon\Omega(X)\lra([0,1],\alpha_R)^I$ are $[0,1]$-Scott continuous, for which we only need to check, by Corollary \ref{Y=S}, that both  $r\colon ([0,1],\geq)^I\lra\Omega(X)_0$ and $s\colon\Omega(X)_0\lra([0,1],\geq)^I$ are Scott continuous. This can be  deduced from the following facts: \begin{itemize} \setlength{\itemsep}{0pt} \item both  $r\colon ([0,1],\geq)^I\lra\Omega(X)_0$ and $s\colon\Omega(X)_0\lra([0,1],\geq)^I$ preserve order; \item $r\circ s=1_X$; and \item $s\circ r\colon ([0,1],\geq)^I\lra([0,1],\geq)^I$ is Scott continuous, by Proposition \ref{YC=C} and Corollary \ref{Y=S}.\end{itemize}

Since $\Sigma([0,1],\alpha_R)^I= \mathbb{K}^I$ by Lemma \ref{Sigma preserves product}, it follows that both $r\colon\mathbb{K}^I\lra\Sigma\Omega(X)$ and $s\colon\Sigma\Omega(X)\lra\mathbb{K}^I$ are continuous. Thus, both
$$1_X=r\circ s\colon X\lra\mathbb{K}^I\lra\Sigma\Omega(X)$$
and
$$1_X=r\circ s\colon\Sigma\Omega(X)\lra\mathbb{K}^I\lra X$$
are continuous and consequently, $X=\Sigma\Omega(X)$.
\end{proof}

\begin{cor}Let $X$ be a separated $[0,1]$-approach space. Then $X$ is injective if and only if   the following conditions are satisfied: \begin{enumerate}[label=\rm(\roman*)] \setlength{\itemsep}{0pt}
	\item  $\Omega (X)$ is a retract of some power of $([0,1],\alpha_R)$ in   $[0,1]\text{-}{\sf Ord}^\uparrow$; \item $X=\Sigma\Omega (X)$. \end{enumerate}
\end{cor}

\begin{proof}
This follows from a combination of Theorem \ref{preserve product} and the following facts: $\mathbb{K}$ is injective, $\Omega(\mathbb{K})=([0,1],\alpha_R)$ and every functor preserves  retracts. \end{proof}

\begin{thm}\label{6.18}
Let $\with$ be a continuous t-norm that satisfies the condition (S). Then the following statements are equivalent:
\begin{enumerate}[label=\rm(\arabic*)] \setlength{\itemsep}{0pt}
	\item $\with$ is isomorphic to the \L ukasiewicz t-norm.
	\item Every continuous $[0,1]$-lattice with the Scott $[0,1]$-approach structure is injective.
\end{enumerate}
\end{thm}

\begin{lem}\label{dlinj}
The $[0,1]$-approach space $\Sigma([0,1],\alpha_L)$ is  injective   if and only if $\with$ is isomorphic to the {\L}ukasiewicz t-norm.
\end{lem}

\begin{proof}
If $\with$ is isomorphic to the {\L}ukasiewicz t-norm, then \[f\colon ([0,1],\alpha_L)\lra([0,1],\alpha_R),\quad f(x)=x\ra 0\] is an isomorphism,   hence $\Sigma([0,1],\alpha_L)$ is injective.

Next, we show that $\Sigma([0,1],\alpha_L)$ is not injective when $\with$ is not isomorphic to the {\L}ukasiewicz t-norm.  We proceed with two cases.

\textbf{Case 1}.   $\with$ does not satisfy the condition (S). Then,  by Example \ref{continuity of dl}, $([0,1],\alpha_L)$ is not continuous, hence $\Sigma([0,1],\alpha_L)$ is not injective by Theorem \ref{preserve product}.

\textbf{Case 2}.  $\with$ satisfies the condition (S). Since $\with$ is not isomorphic to the {\L}ukasiewicz t-norm, there exist idempotent elements $p,q$ of  $([0,1],\with,1)$ such that the restriction of $\with$ on $[p,q]$ is either isomorphic to the G\"{o}del t-norm or to the product t-norm. Suppose on the contrary that $\Sigma([0,1],\alpha_L)$ is injective. Consider the subspace $\{p,q\}$ of $\mathbb{K}$ and the map
$$f\colon \{p,q\}\lra\Sigma([0,1],\alpha_L), \quad \text{$f(p)=q$, $f(q)=p$}.$$ It is clear that $f$  preserves the specialization $[0,1]$-order. Since the $[0,1]$-approach structure on a set with at most two elements is just a $[0,1]$-order, it follows that $f$ is continuous. So, there is a continuous map $\overline{f}\colon\mathbb{K}\lra\Sigma([0,1],\alpha_L)$ that extends $f$. By Lemma \ref{YC=C}, $$\overline{f}\colon ([0,1],\alpha_R)\lra([0,1],\alpha_L)$$ is $[0,1]$-Scott continuous, hence $\overline{f}\colon [0,1]\lra[0,1]$ transforms non-empty meets to joins. Since $\overline{f}(q)=p$ and $\overline{f}$ preserves $[0,1]$-order, it follows that $$\overline{f}(x)\leq\alpha_R(x,q)\ra \overline{f}(q)=x\ra p=p $$ for all $x\in(p,q)$,  hence $\bv_{x>p}\overline{f}(x)=p$, contradicting that $\overline{f}(p)=f(p)=q$. \end{proof}

\begin{rem}That $\Sigma([0,1],\alpha_L)$ is not injective when $\with$ is isomorphic to the product t-norm is known in  \cite[Example 4.14]{GH}   and \cite[Example 5.10]{LiZ18b}. \end{rem}

\begin{proof}[Proof of Theorem \ref{6.18}]
$(1)\Rightarrow(2)$   Every continuous $[0,1]$-lattice $X$ is a retract of some power of $([0,1],\alpha_L)$ in    $[0,1]\text{-}{\sf Ord}^\uparrow$ by Proposition \ref{SCL}, hence a retract of some power of $([0,1],\alpha_R)$ in  $[0,1]\text{-}{\sf Ord}^\uparrow$ since $([0,1],\alpha_L)$ is isomorphic to $([0,1],\alpha_R)$ in this case. Therefore, $\Sigma X$ is injective.

(2) $\Rightarrow$ (1) Since $([0,1],\alpha_L)$ is continuous by Example \ref{continuity of dl}, it follows that the space $\Sigma([0,1],\alpha_L)$ is injective, and consequently,  $\with$ is isomorphic to the \L ukasiewicz t-norm by Lemma \ref{dlinj}.
\end{proof}



We do not know whether or not Theorem \ref{6.18} holds for all continuous t-norms.

\end{document}